\documentclass[11pt]{amsart}
\usepackage{amsfonts, amsmath, amssymb, amsthm, amscd, enumerate}

\usepackage{euscript,amstext}
\usepackage{mathrsfs}
\usepackage{indentfirst}

\usepackage[utf8]{inputenc}
\allowdisplaybreaks

\theoremstyle{plain}
\newtheorem{thm}{Theorem}[section]
\newtheorem{lem}[thm]{Lemma}

\newtheorem{cor}[thm]{Corollary}

\theoremstyle{remark}
\newtheorem{rem}[thm]{Remark}

\theoremstyle{definition}
\newtheorem{exam}[thm]{Example}
\newtheorem{nota}[thm]{Notation}
\newtheorem{dfn}[thm]{Definition}

\newcommand{\bbC}{\mathbb C}
\newcommand{\bbN}{\mathbb N}
\newcommand{\bbR}{\mathbb R}

\newcommand{\cW}{\mathcal W}

\newcommand{\bd}{\mathbf d}

\newcommand{\bs}{\mathbf s}
\newcommand{\bt}{\mathbf t}
\newcommand{\bx}{\mathbf x}

\newcommand{\rB}{B}
\newcommand{\rE}{E}
\newcommand{\rF}{F}
\newcommand{\rG}{G}
\newcommand{\rH}{H}
\newcommand{\rM}{M}

\newcommand{\rU}{U}
\newcommand{\rV}{V}

\newcommand{\rZ}{Z}

\DeclareMathOperator{\Forall}{\forall}

\DeclareMathOperator{\id}{id}
\DeclareMathOperator{\tr}{tr}
\DeclareMathOperator{\op}{op}
\DeclareMathOperator{\str}{str}
\DeclareMathOperator{\HS}{HS}

\DeclareMathOperator{\Ker}{Ker}

\DeclareMathOperator{\Rg}{Rg}
\DeclareMathOperator{\rk}{rk}

\begin{document}

\title[Regularity of torsion form]
{Regularity of analytic torsion form on families of normal coverings}

\author{Bing Kwan SO}
\author{GuangXiang SU}
\thanks{Chern Institute of Mathematics and LPMC, Nankai University, Tianjin, 300071, China. \\
e-mail: bkso@graduate.hku.hk (B.K. So), guangxiangsu@gmail.com (G. Su)}

\begin{abstract}We prove the smoothness of the $L ^2 $-analytic torsion form for fiber bundle with positive Novikov-Shubin invariant.
We do so by generalizing the arguments of Azzali-Goette-Schick to an appropriate Sobolev space,
and proving that the Novikov-Shubin invariant is also positive in the Sobolev setting,
using an argument of Alvarez Lopez-Kordyukov.
\end{abstract}

\maketitle

\section{Introduction}
Let $M$ be a closed Riemannian manifold and $F$ be a flat vector bundle on $M$, Ray and Singer in \cite{RS} introduced the analytic torsion which is the analytic analogue of the combinatorial torsion (cf. \cite{Mi}). Let $\rZ \to \rM \xrightarrow{\pi} \rB $ be a fiber bundle with connected closed fibers $Z_{x}=\pi^{-1}(x)$ and $F$ be a flat complex vector bundle on $M$ with a flat connection $\nabla^{F}$ and a Hermitian metric $h^{F}$. Let $T^{H}M$ be a horizontal distribution for the fiber bundle and $g^{TZ}$ be a vertical Riemannian metric. Then in  \cite{Bismut;AnaTorsion} Bismut and Lott introduced the torsion form $\mathcal{T}(T^{H}M, g^{TZ},h^{F})\in \Omega(B)$ (cf. \cite[(3.118)]{Bismut;AnaTorsion}) defined by
\begin{align}
\label{BLDfn}
\mathcal{T}(T^{H}M,g^{TZ},h^{F})
=- \int_{0}^{+\infty} & \Big[f^{\wedge}(C_{t}',h^{W})- \frac{\chi'(Z;F)}{2}f'(0) \\ \nonumber
&-\Big(\frac{\dim (Z) \rk (F)\chi(Z)}{4}
-\frac{\chi'(Z;F)}{2} \Big) f'\Big(\frac{i\sqrt{t}}{2}\Big) \Big] \frac{dt}{t}.
\end{align}
See \cite{Bismut;AnaTorsion} for the meaning of the terms in the integrand. To show the integral in the above formula is well-defined, it needs to calculate the asymptotic of $f^{\wedge}(C_{t}',h^{W})$ as $t\to 0$ and the asymptotic as $t\to \infty$. For the asymptotic as $t\to 0$, they used the local index technique. For the asymptotic as $t\to \infty$, the key fact is that the fiber $Z$ is closed, so the fiberwise operators involved have uniform positive lower bound for positive eigenvalues. They also proved a $C^{\infty}$-analogue of the Riemann-Roch-Grothendieck theorem and proved that the torsion form is the transgression of the Riemann-Roch-Grothendieck theorem (cf. \cite[Theorem 3.23]{Bismut;AnaTorsion}) and showed the zero degree part of $\mathcal{T}(T^{H}M, g^{TZ},h^{F})\in\Omega(B)$ is the Ray-Singer analytic torsion {(cf. \cite[Theorem 3.29]{Bismut;AnaTorsion})}.

On the other hand, the $L^2$-analytic torsion was defined and studied by several authors, cf. \cite{CM}, \cite{L}, \cite{M} and etc. So it is natural to extend the $L^2$-analytic torsion to the family case, that is to define and study the Bismut-Lott torsion form when the fiber $Z$ is non-compact. From the above we see that it needs to study the asymptotic of the $L^2$ analogue of $f^{\wedge}(C_{t},h^W)$ as $t\to 0$ and $t\to \infty$. Since in the $L^2$ case $f^{\wedge}(C_{t}',h^W)$ has the same asymptotic as $t\to 0$, so this part is easy. But in general the integral at $\infty$ does not converge, since in the $L^2$ case the positive eigenvalues of fiberwise operator involved in $f^{\wedge}(C_{t}',h^{W})$ may not have a positive lower bound. To overcome this difficulty, one considers the special case where the Novikov-Shubin invariant is (sufficiently) positive.
In \cite{Gong;CoverTorsion} Gong and Rothenberg defined the $L^2$-analytic torsion form
and proved that the torsion form is smooth,
under the condition that the Novikov-Shubin invariant is at least half of the dimension of the base manifold.
Heitsch-Lazarov \cite{Heitsch;FoliationTorsion} generalized essentially the same arguments to foliations.
In \cite{Schick;NonCptTorsionEst} Azzali, Goette and Schick proved
that the integrand defining the $L^2$-analytic torsion form,
as well as several other invariants related to the signature operator,
converges provided the Novikov-Shubin invariant is positive (or of determinant class and $L^{2}$-acyclic).
However, they did not prove the smoothness of the $ L ^2 $-analytic torsion form.
To consider transgression formula, they had to use weak derivatives.

The aim of this paper is to establish the regularity of the $L^{2}$-analytic torsion form,
in the case when the Novikov-Shubin invariant is positive.
Our motivation comes from the study of analytic torsion on some ``non-commutative" spaces
(along the lines of \cite{Lott;FoliationInd}, etc., for local index).
In this case one considers universal differential forms (as in \cite{Lott;FoliationInd}),
and the Duhamel's formula for the heat operator having infinitely many terms.
Instead, one makes essential use of the results of \cite{Schick;NonCptTorsionEst} to ensure that
\eqref{BLDfn} is well defined in the non-commutative case.
We achieve this result by generalizing Azzali-Goette-Schick's arguments to some Sobolev spaces.

The rest of the paper is organized as follows. In Section 2,
we define Sobolev norms on the spaces of kernels on the fibered product groupoid.
Unlike \cite{Schick;NonCptTorsionEst},
we consider Hilbert-Schmit type norms on the space of smoothing operators.
Given a kernel, the Hilbert-Schmit norm can be explicitly written down.
As a result,
we are able to take into account derivatives in both the fiber-wise and transverse directions,
with the help of a splitting similar to \cite{Heitsch;FoliationHeat}. In section 3, we turn to prove that having positive Novikov-Shubin invariant
implies positivity of the Novikov-Shubin invariant in the Sobolev settings.
We adapt an argument of Alvarez Lopez-Kordyukov \cite{Lopez;FoliationHeat}. In Section 4,
we apply the arguments in \cite{Schick;NonCptTorsionEst}
and conclude that the integral \eqref{BLDfn} converges in all Sobolev norms, and hence the regularity of the $L^{2}$ analytic torsion form.
\\
\
\\
{\bf Acknowledgement} The authors are very grateful to referees for their very careful reading of the manuscript of the paper and many valuable suggestions. The second author was supported by NSFC 11571183.

\section{Preliminaries}
In this section, we will define Sobolev norms on the space of kernels on the fibered product groupoid. 
\subsection{The geometric settings}
\label{Dfn}
Let $\rZ \to \rM \xrightarrow{\pi} \rB $ be a fiber bundle with connected fibers $Z_x=\pi^{-1}(x)$, $x\in B$. Let $\rE \xrightarrow{\wp} \rM$ be a vector bundle.
We assume $\rB$ is compact.
Let $\rV := \Ker (d \pi ) \subset T \rM$.

We suppose that there is a finitely generated discrete group $\rG$ acting on $\rM$ from the right freely,
properly discontinuously. We also assume that the group $G$ acts on $B$ such that the actions commute with $\pi$ and $\rM _0 := \rM / \rG $ is a compact manifold.
Since the submersion $\pi$ is $\rG$-invariant,
$\rM _0 $ is also foliated and denote such foliation by $\rV _0 $.
Fix a distribution $\rH _0 \subset T \rM _0 $ complementary to $\rV _0$.  Fix a metric on $V_{0}$ and take a $G$-invariant metric on $B$, then these induce a Riemannian metric on $M_{0}$ as $g^{V_0}\oplus \pi^* g^{TB}$ on $TM_0=V_0\oplus H_0$.

Since the projection from $\rM$ to $\rM _0 $ is a local diffeomorphism,
one gets a $\rG$-invariant splitting $T \rM = \rV \oplus \rH $.
Denote by $P ^\rV , P ^\rH$ respectively the projections to $\rV $ and $\rH$.
Moreover, one gets a $\rG$-invariant metric on $\rV$ and a Riemannian metric on $M$ as $g^{TM}=g^{V}\oplus \pi ^*g^{TB}$ on $TM=V\oplus H$. 

Given any vector field $X \in \Gamma ^\infty (T \rB)$,
denote the horizontal lift of $X $ by $X ^\rH \in \Gamma ^\infty (\rH) \subset \Gamma ^\infty (T \rM )$.
By our construction, $| X ^\rH | _{g _\rM} (p) = | X | _{g _\rB } (\pi (p))$.

Denote by $\mu _x , \mu _\rB $ respectively the Reimannian measures on $\rZ _x $ and $\rB $.

\begin{dfn}
Let $\rE \xrightarrow{\wp} \rM$ be a complex vector bundle.
We say that $\rE $ is a contravariant $\rG$-bundle if $\rG$ also acts on $\rE$ from the right,
such that for any $v \in \rE , g \in \rG $, $\wp (v g) = \wp (v) g \in \rM $,
and moreover $\rG$ acts as a linear map between the fibers.

The group $\rG$ then acts on sections of $\rE$ from the left by
$$ s \mapsto g ^* s, \quad (g ^* s ) (p) := s (p g ) g ^{-1} \in \wp ^{-1} (p), \quad \forall \; p \in \rM .$$
\end{dfn}

We assume that $\rE$ is endowed with a $\rG$-invariant metric $g _\rE$,
and a $\rG$-invariant connection $\nabla ^\rE$
(which is obviously possible if $\rE$ is the pullback of some bundle on $\rM _0$).
In particular, for any $G$-invariant section $s$ of $\rE$,
$| s |$ is a $G$-invariant function on $\rM$.

In the following, for any vector bundle $\rF$ we denote its dual bundle by $\rF'$.

Recall that the ``infinite dimensional bundle'' over $\rB$ in the sense of Bismut is a vector bundle with typical fiber
$\Gamma _c ^\infty (\rE |_{\rZ _x } ) $ (or other function spaces) over each $x \in \rB$.
We denote by $\rE _\flat$ for such Bismut bundle.
The space of smooth sections on $\rE _\flat$ is, as a vector space, $\Gamma ^\infty _c (\rE )$.
Each element $s \in \Gamma ^\infty _c (\rE ) $ is regarded as a map
$$ x \mapsto s |_{\rZ _x} \in \Gamma _c ^\infty (\rE |_{\rZ _x } ) , \quad \Forall x \in \rB .$$
In other words, one defines a section on $\rE _\flat $ to be smooth,
if the images of all $x \in \rB$ fit together to form an element in $\Gamma _c ^\infty (\rE )$.
In particular,
$ \Gamma ^\infty _c ((\rM \times \bbC ) _\flat ) = C ^\infty _c (\rM ),$
and one identifies $\Gamma ^\infty _c (T \rB \otimes (\rM \times \bbC ) _\flat ) $
with $\Gamma ^\infty _c (\rH )$ by
$X \otimes f \mapsto f X ^\rH $.

\subsection{Covariant derivatives and Sobolev spaces}


Let $\nabla ^\rE $ be a $\rG$-invariant connection on $\rE$.
Denote by $\nabla ^{T \rM }, \nabla ^{T \rB} $ the Levi-Civita connections
(with respect to the metrics defined in the last section).
Note that $[X ^\rH , Y ] \in \Gamma ^\infty (\rV )$ for any vertical vector field $Y \in \Gamma ^\infty (\rV )$.
One naturally defines the connections 
\begin{align*}
\nabla ^{\rV _\flat} _X Y
:= [X ^\rH , Y ] , \quad & \Forall Y \in \Gamma ^\infty (\rV _\flat ) \cong \Gamma ^\infty (\rV ), \\
\nabla ^{\rE _\flat } _X s := \nabla ^\rE _{X ^\rH } s ,
\quad & \Forall s \in \Gamma ^\infty (\rE _\flat ) \cong \Gamma _{c}^\infty (\rE ).
\end{align*}

\begin{dfn}
\label{DiffDfn1}
The covariant derivative on $\rE _\flat $ is the map
$$\dot \nabla ^{\rE _\flat } : 
\Gamma ^\infty (\otimes ^\bullet T ^* \rB \bigotimes \otimes ^\bullet \rV' _\flat \bigotimes \rE _\flat )
\to \Gamma ^\infty (\otimes ^{\bullet + 1} T ^* \rB \bigotimes \otimes ^\bullet \rV' _\flat \bigotimes \rE _\flat ),$$
defined by
\begin{multline}
\left(\dot \nabla ^{\rE _\flat } s \right)(X _0 , X _1 , \cdots, X _k ; Y _1 , \cdots,Y _l )
:=\nabla ^{\rE _\flat } _{X _0} s ( X _1 , \cdots, X _k ; Y _1 , \cdots, Y _l ) \\
- \sum _{j=1} ^l s (X _1 , \cdots, X _k ; Y _1 , \cdots , \nabla ^{\rV _\flat} _{X _0 } Y _j , \cdots , Y _l ) 
- \sum _{i=1} ^k s (X _1 , \cdots , \nabla ^{T \rB } _{X _0 } X _i , \cdots , X _k ; Y _1 , \cdots Y _l ) ,
\end{multline}
for any $k, l \in \bbN , X _0 , \cdots, X _k \in \Gamma ^\infty (T \rB), Y _1 , \cdots, Y _l \in \Gamma ^\infty (\rV)$.
\end{dfn}

Clearly, taking covariant derivative can be iterated,
which we denote by $(\dot \nabla ^{\rE _\flat }) ^m $, \\ $ m = 1, 2, \cdots$.
Note that $(\dot \nabla ^{\rE _\flat }) ^m $ is a differential operator of order $m$.

Also, we define
$\dot \partial ^{\rV } :
\Gamma ^\infty (\otimes ^\bullet T ^* \rB \bigotimes \otimes ^\bullet \rV' _\flat \bigotimes \rE _\flat )
\to \Gamma ^\infty (\otimes ^\bullet T ^* \rB \bigotimes \otimes ^{\bullet + 1} \rV' _\flat \bigotimes \rE _\flat )$ by
\begin{multline}
\label{VertDiff1}
\left(\dot \partial ^{\rV } s \right)(X _1 , \cdots, X _k ; Y _0 , Y _1 , \cdots, Y _l )
:= \nabla ^\rE _{Y _0} s ( X _1 , \cdots, X _k ; Y _1 , \cdots, Y _l ) \\
- \sum _{j=1} ^l s (X _1 , \cdots, X _k ; Y _1 , \cdots , P ^\rV (\nabla ^{T \rM} _{Y _0 } Y _j) , \cdots , Y _l ) .
\end{multline}

In the following definition,
we regard $(\dot \nabla ^{\rE _\flat})^i (\dot\partial^{V}) ^j s
\in \Gamma ^\infty ( \otimes ^i \rH' \bigotimes \otimes ^j \rV' \bigotimes \rE_{\flat} )$.
\begin{dfn}
\label{SobDfn}
For $s\in \Gamma^{\infty}_{c}(E)$, we define its $m$-th Sobolev norm by
\begin{equation}
\| s \| ^2 _m
:= \sum _{i+j \leq m} \int _{x \in \rB} \int _{y \in \rZ _x }
\big| (\dot \nabla ^{\rE _\flat})^i (\dot \partial ^{\rV} ) ^j s \big| ^2 (x, y)
\mu _x (y) \mu _\rB (x).
\end{equation}
Denote by $\cW ^m (\rE)$ the Sobolev completion of $\Gamma ^\infty _c (\rE ) $ with respect to $\| \cdot \| _m $.
\end{dfn}

Recall that an operator $A$ is called $C^{\infty}$-bounded if in normal coordinates the coefficients and their derivatives are $C^{\infty}$-bounded. 

Since $\rM$ is locally isometric to a compact space $\rM _0$,
it is a manifold with bounded geometry (see \cite[Appendix 1]{S} for an introduction).
Moreover, $\nabla ^\rE$ is a $C ^\infty$-bounded differential operator,
because by $\rG $ invariance the Christoffel symbols of
$\nabla ^\rE$ and all their derivatives are uniformly bounded.
Using normal coordinate charts and parallel transport with respect to $\nabla ^\rE$ as trivialization,
one sees that $\rE$ is a bundle with bounded geometry.

Since the operators $\dot \nabla ^{\rE _\flat } $ and $ \dot \partial ^{\rV } $ are just respectively the
$(0, 1)$ and $(1, 0)$ parts of the usual covariant derivative operator,
our Definition \ref{SobDfn} is equivalent to the standard Sobolev norm \cite[Appendix 1 (1.3)]{S}
(with $p=2$ and $s$ non-negative integers).

One has the elliptic regularity for these Sobolev spaces:
\begin{lem}
\label{EllReg1}
\cite[Lemma 1.4]{S}
Let $A $ be any $C ^\infty$-bounded, uniformly elliptic, differential operator of order $m$.
For any $i, j\geq 0$, there exists a constant $C$ such that for any $s \in \Gamma ^\infty _c (\rE) $
$$ \| s \| _{i + m} \leq C ( \| A s \| _i + \| s \| _j) .$$
\end{lem}

\begin{rem}
Throughout this paper, by an ``elliptic operator" on a manifold,
we mean elliptic in all directions,
without taking any foliation structure into consideration.
We use the term ``fiber-wise elliptic operators" to refer to differential operators that are fiber-wise and elliptic restricted to fibers.
\end{rem}

\subsection{The fibered product}
\begin{dfn}
The fibered product of the manifold $M$ is 
$$\rM \times _\rB \rM := \{ (p, q) \in \rM \times \rM : \pi (p) = \pi (q) \} ,$$
and with the maps from $\rM \times _\rB \rM $ to $\rM$ defined by
$$ \bs (p, q) := q, \quad \bt (p, q) := p .$$
The manifold $\rM \times _\rB \rM $ is a fiber bundle over $\rB $, with typical fiber $\rZ \times \rZ $.
One naturally has the splitting \cite[Section 2]{Heitsch;FoliationHeat}
$$T (\rM \times _\rB \rM) = \hat \rH \oplus \rV _\bt \oplus \rV _\bs ,$$
where
$$\rV _\bs := \Ker (d \bt) , \quad \rV _\bt := \Ker (d \bs ).$$
\end{dfn}
Note that $\rV _\bs \cong \bs ^{*} \rV$ and $\rV _\bt \cong \bt ^{*} (\rV )$.
As in Section 1.1, we endow $\rM \times _\rB \rM $ with a metric by lifting the metrics on $\rH _{0}$ and $\rV_{0}$.
Then $\rM \times _\rB \rM $ is a manifold with bounded geometry.

\begin{nota}
With some abuse in notations,
we shall often write elements in $\rM \times _\rB  \rM $ as a triple $(x, y, z)$
and $\bs (x, y, z) = (x, z), \bt (x, y, z) = (x, y) \in \rM $,
where $x \in \rB , y , z \in \rZ _x $
\end{nota}

Let $\rG$ act on $\rM \times _\rB \rM $ by the diagonal action
$$ (p, q) g := (p g , q g ). $$
Let $\rE \to \rM $ be a contravariant $\rG$-vector bundle and $\rE '$ be its dual.
We shall consider
$$\hat \rE \to \rM \times _\rB \rM := \bt ^* \rE \otimes \bs ^* \rE ' .$$
Given a $\rG$-invariant connection $\nabla ^\rE $ on $\rE $, let
$$\nabla ^{\hat \rE } := \bt ^* \nabla ^{\rE }\otimes {\rm Id}_{\bs^*E'} + {\rm Id}_{\bt ^*E}\otimes \bs ^* \nabla ^{\rE ' }$$
be the tensor product of the pullback connections.
Fix any local base $\{ e _1, \cdots e _r \} $ of $\rE '$ on some $\rU \subset \rM$,
any section can be written as
$$ s = \sum _{i=1} ^r u _i \otimes \bs ^* e _i$$
on $\bs ^{-1} (\rU)$, where $u _i \in \Gamma ^\infty (\bt ^* \rE )$.
Then by definition we have
\begin{equation}
\nabla _X ^{\hat \rE } \left(\sum _{i=1} ^r u _i \otimes \bs ^* e _i\right)
= \sum _{i=1} ^r (\nabla ^{\bt ^* \rE } _X u _i ) \otimes \bs ^* e _i
+ u _i \otimes \bs ^* (\nabla ^{\rE '} _{d \bs (X)} e _i) ,
\end{equation}
for any vector $X$ on $\rM$.

Similar to Definition \ref{DiffDfn1}, we define the covariant derivative operators on \\
$\Gamma ^\infty (\otimes ^\bullet T ^* \rB \bigotimes \otimes ^\bullet (\rV' _\bt)_\flat
\bigotimes \otimes ^\bullet (\rV' _\bs)_\flat \bigotimes \hat \rE _\flat )$.
\begin{dfn}
Define
\begin{multline}
\nonumber
\left(\dot \nabla ^{\hat \rE _\flat } \psi \right)
(X _0 , X _1 , \cdots ,X _k ; Y _1 , \cdots, Y _l , Z _1 , \cdots ,Z _{l'}) \\
:= \nabla ^{\hat \rE _\flat } _{X _0} \psi
( X _1 , \cdots ,X _k ; Y _1 , \cdots ,Y _l , Z _1 , \cdots, Z _{l'}) \\ 
- \sum _{1 \leq j \leq l} \psi (X _1 , \cdots, X _k ;
Y _1 , \cdots , \nabla ^{\rV _\flat} _{X _0 } Y _j , \cdots , Y _l , Z _1 , \cdots, Z _{l'}) \\ 
- \sum _{1 \leq j \leq l'} \psi (X _1 , \cdots, X _k ;
Y _1 , \cdots , Y _l , Z _1 , \cdots , \nabla ^{\rV _\flat} _{X _0 } Z _j , \cdots , Z _{l'}) \\ 
- \sum _{1 \leq i \leq k} \psi (X _1 , \cdots , \nabla ^{T \rB } _{X _0 } X _i , \cdots , X _k ; Y _1 , \cdots Y _l
, Z _1 , \cdots, Z _{l'}), 
\end{multline}
\begin{multline}
\nonumber
\left(\dot \partial ^{\bs } \psi\right) (X _1 , \cdots, X _k ; Y _0 , Y _1 , \cdots, Y _l , Z _1 , \cdots, Z _{l'}) \\
:=\nabla ^{\hat \rE} _{Y _0} \psi ( X _1 , \cdots, X _k ; Y _1 , \cdots ,Y _l , Z _1 , \cdots ,Z _{l'}) \\ 
- \sum _{1 \leq j \leq l} \psi (X _1 , \cdots ,X _k ;
Y _1 , \cdots , P ^{\rV ^\bs} (\nabla ^{T \rM} _{Y _0 } Y _j) , \cdots , Y _l , Z _1 , \cdots, Z _{l'}) \\ 
- \sum _{1 \leq j \leq l'} \psi (X _1 , \cdots ,X _k ;
Y _1 , \cdots ,Y _l , Z _1 , \cdots , P ^{\rV ^\bt} [Y _0 , Z _j] , \cdots , Z _{l'}), \\ 
\end{multline}
\begin{multline}
\nonumber
\left(\dot \partial ^{\bt } \psi \right)(X _1 , \cdots, X _k ;  Y _1 , \cdots ,Y _l , Z _0 , Z _1 , \cdots, Z _{l'}) \\
:= \nabla ^{\hat \rE} _{Y _0} \psi ( X _1 , \cdots, X _k ; Y _1 , \cdots ,Y _l , Z _0 , Z _1 , \cdots, Z _{l'})\\ 
- \sum _{1 \leq j \leq l} \psi (X _1 , \cdots, X _k ;
Y _1 , \cdots , P ^{\rV ^\bs} [Z _0 , Y _j] , \cdots , Y _l , Z _1 , \cdots, Z _{l'}) \\ 
- \sum _{1 \leq j \leq l'} \psi (X _1 , \cdots ,X _k ;
Y _1 , \cdots ,Y _l , Z _1 , \cdots , P ^{\rV ^\bt} (\nabla ^{T \rM} _{Z _0 } Z _j) , \cdots , Z _{l'}).
\end{multline}
\end{dfn}
Given any vector fields $Y, Z \in \rV$.
Let $Y ^\bs , Z ^\bt $ be respectively the lifts of $Y $ and $Z $ to $\rV _\bs $ and $\rV _\bt$.
Then $[Y ^\bs , Z ^\bt ] = 0 $.
It follows that as differential operators,
$$ [\dot \partial ^\bs , \dot \partial ^\bt ] = 0 .$$
Also, it is straightforward to verify that
$$ [\dot \nabla ^{\hat \rE _\flat} , \dot \partial ^\bs ]
\text{ \and } [\dot \nabla ^{\hat \rE _\flat} , \dot \partial ^\bt ] $$
are both zeroth order differential operators (i.e. smooth bundle maps).

Fix a local trivialization
$$ \bx _\alpha : \pi ^{-1} (\rB _\alpha ) \to \rB _\alpha \times \rZ , \quad
p \mapsto (\pi (p) , \varphi ^\alpha (p)),$$
where $\rB = \bigcup _{\alpha } \rB _\alpha $ is a finite open cover (since $B$ is compact), 
and $\varphi ^\alpha |_{\pi ^{-1} (x)} : \rZ _x \to \rZ $ is a diffeomorphism.
Such a trivialization induces a local trivialization of the fiber bundle $\rM \times _\rB \rM \xrightarrow{\bt} \rM $ by
$\rM = \bigcup \rM _{\alpha } , \rM _\alpha := \pi ^{-1} (\rB _\alpha ) $,
$$ \hat \bx _\alpha : \bt ^{-1} (\rM _\alpha ) \to \rM _\alpha \times \rZ , \quad
(p, q) \mapsto ( p , \varphi ^\alpha (q)).$$
On $\rM _\alpha \times \rZ$ the source and target maps are explicitly given by
\begin{equation}
\label{LocalGpoid}
\bs \circ (\hat \bx _\alpha) ^{-1} (p, z) = (\bx _\alpha ) ^{-1} (\pi (p) , z)
\text{ and } \bt \circ (\hat \bx _\alpha) ^{-1} (p, z) = p .
\end{equation}

For such trivialization, one has the natural splitting
$$ T (\rM _\alpha \times \rZ) = \rH ^\alpha \oplus \rV ^\alpha \oplus T \rZ ,$$
where $\rH ^\alpha $ and $ \rV ^\alpha $ are respectively $\rH $ and $\rV$ restricted to $\rM _\alpha \times \{ z \}$,
$z \in \rZ $.
It follows from (\ref{LocalGpoid}) that
$$\rV ^\alpha = d \hat \bx _\alpha (\rV _\bs), \quad T \rZ = d \hat \bx _\alpha (\rV _\bt) .$$
Given any vector field $X$ on $\rB$,
let $X ^\rH , X ^{\hat \rH}$ be respectively the lifts of $X$ to $\rH $ and $\hat \rH $.
Since $d \bt (X ^{\hat \rH} ) = d \bs (X ^{\hat \rH} ) = X ^\rH $,
it follows that
$$d \hat \bx _\alpha (X ^{\hat \rH }) = X ^{\rH ^\alpha } + d \varphi ^\alpha (X ^\rH ).$$
Note that $d \varphi ^\alpha (X ^\rH ) \in T \rZ \subseteq T (\rM _\alpha \times \rZ) $.

Corresponding to the splitting $T (\rM _\alpha \times \rZ) = \rH ^\alpha \oplus \rV ^\alpha \oplus T \rZ$,
one can define the covariant derivative operators.
Let $\nabla ^{T \rM _\alpha } $ be the Levi-Civita connection on $\rM_\alpha$
and $\nabla^{TZ}$ be the Levi-Civita connection on $\rZ$.
Define for any smooth section
$\phi \in \Gamma ^\infty (\otimes ^\bullet T ^* \rB \bigotimes \otimes ^\bullet (\rV ^\alpha)' _\flat
\bigotimes \otimes ^\bullet T ^* \rZ _\flat \bigotimes (\hat \bx _\alpha ^{-1} )^* \hat \rE _\flat )$,
\begin{multline}
\nonumber
\left(\dot \nabla ^{\alpha } \phi \right)
(X _0 , X _1 , \cdots,  X _k ; Y _1 , \cdots, Y _l , Z _1 , \cdots ,Z _{l'}) \\
:= (\bx _\alpha ^* \nabla ^{\hat \rE _\flat }) _{X ^{\rH ^\alpha } _0}
\phi ( X _1 , \cdots ,X _k ; Y _1 , \cdots, Y _l , Z _1 , \cdots ,Z _{l'}) \\ 
- \sum _{1 \leq j \leq l} \phi (X _1 , \cdots ,X _k ; Y _1 , \cdots , [X _0 ^{\rH ^\alpha} , Y _j] , \cdots , Y _l ,
, Z _1 , \cdots, Z _{l'}) \\ 
- \sum _{1 \leq j \leq l'} \phi (X _1 , \cdots, X _k ;
Y _1 , \cdots , Y _l , Z _1 , \cdots , [X _0 ^{\rH ^\alpha} Z _j ], \cdots , Z _{l'}) \\ 
- \sum _{1 \leq i \leq k} \phi (X _1 , \cdots , \nabla ^{T \rB } _{X _0 } X _i , \cdots , X _k ;
Y _1 , \cdots ,Y _l , Z _1 , \cdots, Z _{l'}), 
\end{multline}
\begin{multline}
\nonumber
\left(\dot \partial ^{\alpha } \phi\right) (X _1 , \cdots, X _k ; Y _0 , Y _1 , \cdots ,Y _l , Z _1 , \cdots ,Z _{l'}) \\
:=(\bx _\alpha ^* \nabla ^{\hat \rE _\flat}) _{Y _0} \phi ( X _1 , \cdots, X _k ;
Y _1 , \cdots ,Y _l , Z _1 , \cdots, Z _{l'}) \\ 
- \sum _{1 \leq j \leq l} \phi (X _1 , \cdots ,X _k ;
Y _1 , \cdots , P ^{\rV ^\alpha} (\nabla ^{T \rM _\alpha}_{Y _0 } Y _j), \cdots, Y _l, Z _1 , \cdots ,Z _{l'})\\ 
- \sum _{1 \leq j \leq l'} \phi (X _1 , \cdots ,X _k ;
Y _1 , \cdots ,Y _l , Z _1 , \cdots , P ^{T \rZ} [Y _0 , Z _j] , \cdots , Z _{l'}), 
\end{multline}
\begin{multline}
\nonumber
\left(\dot \partial ^{\rZ } \phi \right)(X _1 , \cdots , X _k ; Y _1 , \cdots ,Y _l , Z _0 , Z _1 , \cdots, Z _{l'}) \\
:=(\bx _\alpha ^* \nabla ^{\hat \rE _\flat} )_{Z _0} \phi ( X _1 , \cdots ,X _k ;
Y _1 , \cdots ,Y _l , Z _0 , Z _1 , \cdots, Z _{l'}) \\ 
- \sum _{1 \leq j \leq l} \phi (X _1 , \cdots, X _k ;
Y _1 , \cdots , P ^{\rV ^\alpha} [Z _0 , Y _j] , \cdots , Y _l , Z _1 , \cdots ,Z _{l'}) \\ 
- \sum _{1 \leq j \leq l'} \phi (X _1 , \cdots, X _k ;
Y _1 , \cdots ,Y _l , Z _1 , \cdots , \nabla ^{T \rZ} _{Z _0 } Z _j , \cdots , Z _{l'}).
\end{multline}

Consider the special case when $\phi = u \otimes \bs ^* e$,
where
$u \in \Gamma ^\infty (\otimes ^\bullet T ^* \rB \bigotimes \otimes ^\bullet (\rV ^\alpha)' _\flat \otimes \bt ^* \rE)$,
$e \in \Gamma ^\infty (\otimes ^\bullet T ^* \rZ _\flat \otimes \rE' )$.
\begin{lem}
\label{TensorD}
For $(x, y, z) \in \rM _\alpha \times \rZ$, one has
\begin{align*}
\dot \nabla ^\alpha (u \otimes \bs ^* e) (x, y, z)
=(\dot \nabla ^{\rE _\flat} u |_{\rM _\alpha \times \{ z \}} (x, y)) \otimes \bs ^* (e (x, z))
+ u \otimes \bs ^* (\nabla ^{\rE' _\flat} e (x, z))
\end{align*}
and 
\begin{align*}
\dot \partial ^{\alpha } (u \otimes \bs ^* e ) (x, y, z)
=(\dot \partial ^\rV u |_{\rM _\alpha \times \{ z \}} (x, y)) \otimes \bs ^* (e (x, z)).
\end{align*}
\end{lem}
\begin{proof}
It suffices to consider the case when $Y _j, Z _{j'}$ are respectively vector fields on $\rM _\alpha $ and $\rZ $
lifted to $\rM _\alpha \times \rZ $.
From this assumption it follows that
$[ Y _j , Z _{j'}] = [X _0 ^{\rH ^\alpha } , Z _{j'} ] = 0$.
The lemma follows by a simple computation.
\end{proof}

We express the (pullback of) the covariant derivatives
$\dot \nabla ^{\hat \rE \flat} \psi, \dot \partial ^\bs \psi$ and $ \dot \partial ^\bt \psi$
in terms of $\dot \nabla ^\alpha \psi ^\alpha$, $\dot \partial ^{\alpha} \psi ^\alpha$
and $\dot \partial ^\rZ \psi ^\alpha$, where $\psi ^\alpha := (\bx _\alpha ^{-1} )^* \psi $.
One directly verifies
\begin{multline}
\left(\dot \nabla ^{\rE _\flat } \psi \right)(X _0 , X _1 , \cdots , X _k ; Y _1 , \cdots, Y _l , Z _1 , \cdots , Z _{l'})\\ 
= (\bx _\alpha ^{-1} )^* \left(  \nabla ^\alpha _{(X ^{\rH ^\alpha} _0 + d \varphi ^\alpha (X ^{\rH } _0 ))}
\psi ^\alpha ( X _1, \cdots , X _k ; d \bx _\alpha (Y _1 , \cdots, Y _l , Z _1 , \cdots , Z _{l'}))\right. \\
- \sum _{1 \leq j \leq l}  \psi ^\alpha \left(X _1 , \cdots, X _k ;
d \bx _\alpha Y _1 , \cdots , [X _0 ^{\rH ^\alpha }, d \bx _\alpha Y _j ],
\cdots , d \bx _\alpha Y _l , d \bx _\alpha (Z _1 , \cdots , Z _{l'})\right) \\ 
- \sum _{1 \leq j \leq l'}  \psi ^\alpha (X _1 , \cdots ,X _k ;
d \bx _\alpha (Y _1 , \cdots , Y _l) , d \bx _\alpha Z _1 , 
 \cdots , [ X _0 ^{\rH ^\alpha} + d \varphi ^\alpha (X ^{\rH } _0 ) , d \bx _\alpha Z _j ] ,
\cdots , d \bx _\alpha Z _{l'}) \\ 
- \left.\sum _{1 \leq i \leq k}  \psi ^\alpha (X _1 , \cdots , \nabla ^{T \rB } _{X _0 } X _i , \cdots , X _k ;
Y _1 , \cdots , Y _l , Z _1 , \cdots , Z _{l'}) \right) \\ 
= (\bx _\alpha ^{-1} )^* \left(  \dot \nabla ^\alpha \psi ^\alpha (X _0 , X _1 , \cdots , X _k ;
Y _1 , \cdots, Y _l , Z _1 , \cdots , Z _{l'})\right. \\ 
+ \dot \partial ^\rZ \psi ^\alpha ( X _1 , \cdots , X _k ;
Y _1 , \cdots, Y _l , d \varphi ^\alpha (X ^{\rH } _0 ), Z _1 , \cdots , Z _{l'}) \\ 
+ \sum _{1 \leq j \leq l'}  \psi ^\alpha \left(X _1 , \cdots, X _k ;
d \bx _\alpha (Y _1 , \cdots , Y _l) , d \bx _\alpha Z _1 , 
 \cdots , (\nabla ^{T \rZ} d \varphi ^\alpha (X ^{\rH } _0 )) ( d \bx _\alpha Z _j ) ,
\cdots , d \bx _\alpha Z _{l'}) \right).
\end{multline}
By similar computations for $\dot \partial ^\bs $ and $\dot \partial ^\bt $, one gets:
\begin{multline}
\left(\dot \partial ^\bs \psi \right)(X _1, \cdots , X _k ; Y _0 , Y _1 , \cdots Y _l , Z _1 , \cdots Z _{l'}) \\ 
= (\bx _\alpha ^{-1} )^*  \big(\dot \partial ^\alpha \psi ^\alpha
\label{Convert3}
(X _1, \cdots, X _k ; d \bx _\alpha (Y _0 , Y _1 , \cdots ,Y _l , Z _1 , \cdots, Z _{l'}) \big), 
\end{multline}
\begin{multline}
\left(\dot \partial ^\bt \psi \right)(X _1, \cdots , X _k ; Y _1 , \cdots, Y _l , Z _0 , Z _1 , \cdots, Z _{l'}) \\ 
= (\bx _\alpha ^{-1} )^*  \big(\dot \partial ^\rZ
\psi ^\alpha ( X _1, \cdots ,X _k ; d \bx _\alpha (Y _1 , \cdots ,Y _l , Z _0 , \cdots ,Z _{l'})) \\ 
+ \sum _{1 \leq j \leq l'}  \psi ^\alpha (X _1 , \cdots, X _k ;
d \bx _\alpha (Y _1 , \cdots ,Y _l ) , d \bx _\alpha Z _1 , \\ 
 \cdots ,
(\nabla ^{T \rZ} _{d \bx _\alpha Z _0 } d \bx _\alpha Z _j
- d \bx _\alpha ( P ^{\rV _\bt} \nabla ^{T \rM } _{Z _0} Z _j),
\cdots , d \bx _\alpha Z _{l'}) \big).
\end{multline}

\subsection{Smoothing operators}
For any $(x, y, z) \in \rM \times _\rB \rM $,
let $\bd (x, y, z) $ be the Riemannian distance between $y, z \in \rZ _x$.
We regard $\bd $ as a continuous, non-negative function on $\rM \times_\rB \rM$.

\begin{dfn}
\label{NWX}
(See  \cite{NWX;GroupoidPdO}). As a vector space,
$$\Psi ^{- \infty } _\infty (\rM \times _\rB \rM , \rE ) :=
\left\{
\begin{array}{ll}
& \text{For any } m \in \bbN, \varepsilon > 0 , \exists C _m > 0  \\
\psi \in \Gamma ^\infty (\hat \rE ) : &  \text{such that } \Forall i+j+k \leq m, \\
& |(\dot \nabla ^{\hat \rE _\flat } )^i (\dot \partial ^{\rV _\bs} )^j (\dot \partial ^{\rV _\bt} )^k \psi |
\leq C _m e ^{- \varepsilon \bd}.
\end{array}
\right\}.
$$
The convolution product structure on $\Psi ^{- \infty } _\infty (\rM \times _\rB \rM , \rE ) $ is defined by
$$ \psi _1 \star \psi _2 (x, y, z)
:= \int _{\rZ _x } \psi _1 (x, y, w) \psi _2 (x, w , z ) \mu _{x } (w) .$$
\end{dfn}

We introduce a Sobolev type generalization of the Hilbert-Schmit norm on
$\Psi ^{- \infty } _\infty (\rM \times _\rB \rM , \rE ) ^\rG $, the space of $\rG$-invariant kernels. Since $G$ is a finitely generated discrete group and acts on $M$ freely, properly discontinuously, then there exists
a smooth compactly supported function $\chi \in C ^\infty _c (\rM )$,
such that
$$ \sum _{g \in \rG } g ^* \chi = 1 .$$

In particular, one may construct $\chi$ as follows.
Denote by $\pi _\rG$ the projection $\rM \to \rM _0 = \rM / \rG$.
There exists some $r > 0$, and a finite collection of geodesic balls $B (p _\alpha , r)$ of radius $r$,
such that $B (p _\alpha , r)$ is diffeomorphic to its image in $\rM _0$ under $\pi _\rG$,
and moreover $\{ B (p _\alpha , \frac{r}{3}) \} $ covers $\rM _0$ (since $\rM _0$ is compact).
Since $\rG$ acts on $\rM $ by isometry, 
$\pi _\rG (B (p _\alpha g , r)) = \pi _\rG (B (p _\alpha , r))$ for all $g \in \rG$.
Thus one may without loss of generality assume that $B (p _\alpha , r)$ are mutually disjoint.

Define the functions $f \in C ^\infty (\bbR), F _\alpha , F \in C ^\infty _c (\rM ) $ by
\begin{align*}
f (t) :=& e ^{- \frac{1}{t ^2}} \text { if } t > 0 , \quad 0 \text { if } t \leq 0 , \\
F _\alpha (p) :=& f \big( 1 - \frac{2 \bd (p , p _\alpha )}{r} \big)
\Big( f \big( \frac{3 \bd (p , p _\alpha )}{r} - 1 \big) 
+ f \big( 1 - \frac{2 \bd (p , p _\alpha )}{r} \big) \Big) ^{-1},
\quad p \in \rM ,\\
F :=& \sum _\alpha F _\alpha .
\end{align*}
Note that $F $ is well defined because $ F _\alpha $ is supported on $B (p _\alpha , r)$,
which is locally finite.
Since by construction 
$$ \big \{ \bigcup _\alpha B (p _\alpha g , \frac {r}{3}) \big \} _{g \in \rG}$$
is a locally finite cover of $\rM $,
$ \sum _g g ^* F $ is also well defined.
Define 
$$ \chi := F ( \sum _g g ^* F ) ^{-1} .$$
Then clearly $\chi $ is the required partition of unity.
Moreover, observe that $\chi ^{\frac{1}{2}}$ is a smooth function because $f ^{\frac{1}{2}}$ is smooth and
all denominators are uniformly bounded away from $0$.

For any $\rG$-invariant $\psi \in \Psi ^{- \infty} _\infty (\rM \times _\rB \rM , \rE ) ^\rG$,
recall that the standard trace of $\psi $ is
$$ \tr _\Psi (\psi ) (x) := \int _{z \in \rZ _x} \chi (x, z) \tr (\psi (x, z, z)) \mu _x (z)
\in C ^\infty (\rB) .$$
The definition does not depend on the choice of $\chi$.
The corresponding Hilbert-Schmit norm is
\begin{align}
\label{0HS}
\int _\rB \big( & \tr _\Psi (\psi \psi ^* ) (x) \big) ^2 \mu _\rB (x) \\ \nonumber
&= \int _\rB \int _{\rZ _x} \chi (x, z) \int _{\rZ _x} \tr (\psi (x, z, y) \psi ^* (x, y, z))
\mu _x (y) \mu _x (z) \mu _\rB (x).
\end{align}
Note that Equation \eqref{0HS} coincides with the $L ^2$-norm of $\psi$.
Generalizing Equation \eqref{0HS} to taking into account derivatives, we define:
\begin{dfn}
The $m$-th Hilbert-Schmit norm on $\Psi ^{- \infty } (\rM \times _\rB \rM , \rE ) ^\rG $ is defined to be
\begin{align*}
\| \psi \| ^2 _{\HS m}
:= \sum _{i+j+k \leq m} \int _\rB \int _{\rZ _x } \chi (x, z) \int _{\rZ _x } \big|
(\dot \nabla ^{\hat \rE _\flat } )^i (\dot \partial ^\bs )^j (\dot \partial ^\bt)^k \psi \big| ^2
& (x, y, z) \mu _x (y) \mu _x (z) \mu _\rB (x),
\end{align*}
for any $\rG$-invariant element $\psi $.
Let $\bar \Psi ^{-\infty} _m (\rM \times _\rB \rM , \rE ) ^\rG$ be the completion of
$ \Psi ^{-\infty} _\infty (\rM \times _\rB \rM , \rE ) ^\rG$ with respect to $\| \cdot \| _{\HS m } $.
\end{dfn}

Similar to Lemma \ref{EllReg1}, one has the elliptic regularity for the Hilbert-Schmit norm:
\begin{lem}
\label{EllReg2}
Let $A$ be a $\rG$-invariant, first order elliptic differential operator,
then for any $m = 0, 1, \cdots$, there exists a constant $C > 0$ such that
$$ \| \psi \| _{\HS m+1} \leq C (\| A \psi \| _{\HS m} + \| \psi \| _m ),$$
for all $\psi \in \Psi ^{- \infty } (\rM \times _\rB \rM , \rE ) ^\rG $.
\end{lem}
\begin{proof}
Define
$$ S := \{ g \in \rG : \chi (g ^* \chi ) \neq 0 \}.$$
Then $S$ is finite because $\{ g ^* \chi \} $ is a locally finite partition of unity.

Consider $(\chi (x, z))^{\frac{1}{2}} \psi $.
By Leibniz rule, one has
$$ (\dot \nabla ^{\hat \rE _\flat } )^i (\dot \partial ^\bs )^j (\dot \partial ^\bt)^k
\chi ^{\frac{1}{2}} \psi
= \chi ^{\frac{1}{2}} (\dot \nabla ^{\hat \rE _\flat } )^i (\dot \partial ^\bs )^j (\dot \partial ^\bt)^k \psi $$
modulo terms involving lower derivatives in $\psi$.
Since $(\chi (x, z))^{\frac{1}{2}} $ is smooth with bounded derivatives,
there exists some $C _1 > 0$ such that for any $(x, y, z) \in \rM \times _\rB \rM$,
\begin{align}
\label{LeibEst1}
\Big|
\sum _{i+j+k \leq m} & \big|
(\dot \nabla ^{\hat \rE _\flat } )^i (\dot \partial ^\bs )^j (\dot \partial ^\bt)^k
\chi ^{\frac{1}{2}} \psi \big| ^2
- \chi \sum _{i+j+k \leq m} \big| (\dot \nabla ^{\hat \rE _\flat } )^i
(\dot \partial ^\bs )^j (\dot \partial ^\bt)^k \psi \big| ^2 \Big| (x, y, z) \\ \nonumber
&\leq \sum _{g \in S} g ^* \chi \Big( C _1 \sum _{i+j+k \leq m-1} \big|
(\dot \nabla ^{\hat \rE _\flat } )^i (\dot \partial ^\bs )^j (\dot \partial ^\bt)^k \psi \big| ^2 \Big) (x, y, z).
\end{align}
Similarly, since $A \chi ^{\frac{1}{2}} - \chi ^{\frac{1}{2}} A $ is a $C ^\infty$-bounded tensor,
one has
\begin{align}
\label{LeibEst2}
\Big|
\sum _{i+j+k \leq m} & \big|
(\dot \nabla ^{\hat \rE _\flat } )^i (\dot \partial ^\bs )^j (\dot \partial ^\bt)^k
(A \chi ^{\frac{1}{2}} \psi ) \big| ^2
- \chi \sum _{i+j+k \leq m} \big| (\dot \nabla ^{\hat \rE _\flat } )^i
(\dot \partial ^\bs )^j (\dot \partial ^\bt)^k A \psi \big| ^2 \Big| \\ \nonumber
&\leq \sum _{g \in S} g ^* \chi \Big( C _2 \sum _{i+j+k \leq m} \big|
(\dot \nabla ^{\hat \rE _\flat } )^i (\dot \partial ^\bs )^j (\dot \partial ^\bt)^k \psi \big| ^2 \Big) .
\end{align}
Since the integrand is $\rG$-invariant, for any $g \in \rG$
$$ \int _{\rM \times _\rB \rM }
g ^* \chi \sum _{i+j+k \leq m-1} \big|
(\dot \nabla ^{\hat \rE _\flat } )^i (\dot \partial ^\bs )^j (\dot \partial ^\bt)^k \psi \big| ^2
\mu _x (y) \mu _x (z) \mu _\rB (x) = \| \psi \| ^2 _{\HS m-1} .$$
Observe that $A$ being $\rG$-invariant implies $A$ is uniformly elliptic and $C ^\infty$-bounded.
Therefore applying Lemma \ref{EllReg1} for $(\chi (x, z))^{\frac{1}{2}} \psi $, there exists constant $C_3>0$ such that
\begin{align*}
\int _{\rM \times _\rB \rM } & \sum _{i+j+k \leq m+1} \big|
(\dot \nabla ^{\hat \rE _\flat } )^i (\dot \partial ^\bs )^j (\dot \partial ^\bt)^k (\chi ^{\frac{1}{2}} \psi ) \big| ^2
\mu _x (y) \mu _x (z) \mu _\rB (x) \\
\leq & C _3 \Big( \int _{\rM \times _\rB \rM } \sum _{i+j+k \leq m} \big|
(\dot \nabla ^{\hat \rE _\flat } )^i (\dot \partial ^\bs )^j (\dot \partial ^\bt)^k (A \chi ^{\frac{1}{2}} \psi ) \big|^2
\mu _x (y) \mu _x (z) \mu _\rB (x) \\
&+ \int _{\rM \times _\rB \rM } \sum _{i+j+k \leq m} \big|
(\dot \nabla ^{\hat \rE _\flat } )^i (\dot \partial ^\bs )^j (\dot \partial ^\bt)^k (\chi ^{\frac{1}{2}} \psi ) \big|^2
\mu _x (y) \mu _x (z) \mu _\rB (x) \Big).
\end{align*}
Then by Equations \eqref{LeibEst1} and \eqref{LeibEst2}, we get the lemma.
\end{proof}

\subsection{Fiber-wise operators}
We turn to consider another class of operators and a different norm.
\begin{dfn}
A fiber-wise operator is a linear operator $A : \Gamma ^\infty _c (\rE _\flat) \to \cW ^0 (\rE)$
such that for all $x \in \rB$,
and any sections $s _1, s _2 \in \Gamma ^\infty _c (\rE _\flat )$,
$$ (A s _1 )(x) = (A s _2) (x), $$
whenever $s _1 (x) = s _2 (x) $.

We say that $A $ is smooth if
$A (\Gamma ^\infty _c (\rE )) \subseteq \Gamma ^\infty (\rE )$.
A smooth fiber-wise operator $A $ is said to be bounded of order $m$
if $A $ extends to a bounded map from $\cW ^m (\rE )$ to itself.

Denote the operator norm of $A : \cW ^m (\rE ) \to \cW ^m (\rE ) $ by $\| A \| _{\op m} $.
\end{dfn}

\begin{exam}
\label{SmoothingExam}
An example of smooth fiber-wise operators are $\Psi ^{- \infty } _\infty (\rM \times _\rB \rM , \rE )$,
acting on $\cW ^m (\rE )$ by vector representation, i.e.
$$ \big( \varPsi s \big) (x, y) := \int _{\rZ _x} \psi (x, y, z ) s (x, z) \mu _x (z) .$$
\end{exam}

\begin{nota}
For the fiber-wise operator $A : \Gamma ^\infty _c (\rE _\flat) \to \cW ^0 (\rE)$ which is of the form given by Example
\ref{SmoothingExam}, we denote its kernel by $A (x, y, z)$.
We shall write
$$ \| A \| _{\HS m} := \| A (x, y, z) \| _{\HS m} ,$$
provided $A (x, y, z) \in \bar \Psi ^{-\infty} _m (\rM \times _\rB \rM , \rE )$.
\end{nota}

The following lemma enables one to construct more fiber-wise operators:
\begin{lem}
Let $A$ be any first order, $C ^\infty$-bounded differential operator on $\rM $ and
$\varPsi \in \Psi ^{- \infty } _\infty (\rM \times _\rB \rM , \rE )$ be as in Example \ref{SmoothingExam}.
Then $[A, \varPsi ] $ is a fiber-wise operator in $\Psi ^{- \infty } _\infty (\rM \times _\rB \rM , \rE )$.
\end{lem}
\begin{proof}
Since multiplication by a tensor or differentiation along $\rV$ is fiber-wise,
it remains to consider operators of the form $\nabla ^{\rE} _{X ^\rH} $, for some vector field $X $ on $\rB$. Let $L^{\nabla^{E}}_{X^{H}}=d^{\nabla^{E}}i_{X^{H}}+i_{X^{H}}d^{\nabla^{E}}$, where $d^{\nabla^{E}}$ is the twisted de Rham operator.  In the following of this paper, the Lie derivatives are all defined in this way.

Let $s \in \Gamma ^\infty _c (\rE)$ be arbitrary.
We first suppose that $\rZ$ is orientable and $\mu _x$ is a volume form.
By the decay condition in Definition \ref{NWX}, one can differentiate under the integral sign to get
\begin{align*}
A \varPsi s (x, z)
=& \int _{\rZ _x} L ^{\nabla ^{\hat \rE }} _{X ^{\hat \rH}} (\psi (x, y, z) s (x, y) \mu _x (y)) \\
=& \int _{\rZ _x} \big( L ^{\nabla ^{\hat \rE }} _{X ^{\hat \rH}} \psi (x, y, z) \big) s (x, y) \mu _x (y)
+ \int _{\rZ _x} \psi (x, y, z) \big( L ^{\nabla ^{\rE }} _{X ^\rH} s (x, y) \big) \mu _x (y) \\
&+ \int _{\rZ _x} \psi (x, y, z) s (x, y) (L ^{\nabla ^{\rE }} _{X ^\rH} \mu _x (y)).
\end{align*}
The second term in the last line is just $\varPsi A s $. Hence the result.

For the general case, one can take a suitable partition of unity and integrate over local volume forms,
then one obtains a similar equation.
\end{proof}

Let $A $ be a smooth fiber-wise operator on $\Gamma ^\infty _c (\rE _\flat )$.
Then $A $ induces a fiber-wise operator $\hat A$ on $\Gamma ^\infty _c (\hat \rE _\flat)$ by
\begin{equation}
\label{FiberwiseOp}
\hat A (u \otimes \bs ^* e) := A ( u |_{ \rM _\alpha \times \{ z \}}) \otimes (\bs ^* e)
\end{equation}
on $\bt ^{-1} (\rM _\alpha ) \cong \rM _\alpha \times \rZ $,
for any sections
$e \in \Gamma ^\infty (\rE ') , u \in \Gamma ^\infty (\bt ^* \rE ) $
and $\psi = u \otimes \bs ^* e \in \Gamma ^{ \infty } _c (\hat \rE ) $.

Note that $\hat A$ is independent of trivialization since $A$ is fiber-wise,
and for any $\alpha , \beta $ and $z \in \rZ$,
the transition function $\bx _\beta \circ (\bx _\alpha ) ^{-1}$ maps the sub-manifold
$\rZ _x \times \{ z \} $ to $\rZ _x \times \{ \varphi ^\beta _x \circ (\varphi ^\alpha _x) ^{-1} (z) \}$
as the identity diffeomorphism.

\subsection{The main theorem}
Suppose that $A $ is smooth and bounded of order $m$ for all $m \in \bbN$.
Consider the covariant derivatives of $\hat A \psi ^\alpha$.

\begin{thm}
\label{AEstCor1}
For any smooth bounded $\rG$-invariant operator $A $, there exist constants $C' _{1,1} , C' _{0, 0} > 0 $ such that for any  $\psi \in \Psi ^{- \infty} _\infty (\rM \times _\rB \rM ) ^\rG $, one has 
$\hat A \psi \in \Psi ^{- \infty} _\infty (\rM \times _\rB \rM ) ^\rG $ and 
$$ \| \hat A \psi \| _{\HS 1} \leq (C ' _{1,1} \| A \|_{\op 1} + C ' _{1,0} \| A \| _{\op 0}) \| \psi \| _{\HS 1}.$$

\end{thm}

\begin{proof}
Fix a partition of unity $\{ \theta _\alpha \} \in C ^\infty _c (\rB)$ subordinate to $\{ \rB _\alpha \}$.
We still denote by $\{ \theta _\alpha \}$ its pullback to $\rM$ and $\rM \times _\rB \rM $.
Fix any Riemannian metric on $\rZ$ and denote the corresponding Riamannian measure by $\mu _\rZ$.
Then one writes
$$ (\hat \bx _\alpha ) _\star (\mu _x \mu _\rB ) = J _\alpha \mu _\rB \mu _\rZ ,$$
for some smooth positive function $J _\alpha $.
Moreover, over any compact subsets on $\rB _\alpha \times \rZ$, $ \frac{1}{J _\alpha} $ is bounded.

Given any $\psi \in \Psi ^{-\infty} _\infty (\rM \times _\rB \rM) ^\rG$,
let $\psi ^\alpha := \hat \bx _\alpha ^* (\psi ) $.
The theorem clearly follows from the inequalities
\begin{align}
\label{EstLem1}
\int _{\rB _\alpha} \int _{\rZ _x} \chi (x, z) \int _{\rZ _x}
| \dot \nabla ^\alpha \hat A (\theta _\alpha \psi ^\alpha ) | ^2 &
\mu _x (y) \mu _x (z) \mu _\rB (x) \\ \nonumber
\leq & (C _1 \| A \| ^2 _{\op 1} + C _2 \| A \| ^2 _{\op 0} ) \| \psi \| ^2 _{\HS 1}, \\
\label{EstLem2}
\int _{\rB _\alpha} \int _{\rZ _x} \chi (x, z) \int _{\rZ _x}
| \dot \partial ^\alpha \hat A (\theta _\alpha \psi ^\alpha ) | ^2 &
\mu _x (y) \mu _x (z) \mu _\rB (x) \\ \nonumber
\leq & (C _1 \| A \| ^2 _{\op 1} + C _2 \| A \| ^2 _{\op 0} ) \| \psi \| ^2 _{\HS 1}, \\
\label{EstLem3}
\int _\rB \int _{\rZ _x} \chi (x, z) \int _{y \in \rZ _x}
| \dot \partial ^\rZ \hat A (\theta _\alpha \psi ^\alpha ) | ^2 &
\mu _x (y) \mu _x (z) \mu _\rB (x)
\leq \| A \| ^2 _{\op 0} \| \psi \| ^2 _{\HS 1} .
\end{align}

Let $\rZ = \bigcup _\lambda \rZ _\lambda $ be a locally finite cover.
Then the support of $\chi \theta _\alpha $ lies in some finite sub-cover.
Let $\chi _\alpha $ be the characteristic function
$$ \chi _\alpha (x, z) = 1 \text{ if } (\chi \theta _\alpha ) (x, z) > 0, \quad 0 \text{ otherwise.}$$
Without loss of generality we may assume $\rE' | _{\rZ _\lambda } $ are all trivial.
For each $\lambda $ fix an orthonormal basis $\{ e ^\lambda _r \}$ of $\rE ' |_{ \rB _\alpha \times \rZ _\lambda }$,
and write
$$\psi ^\alpha := \sum _r u ^\lambda _r \otimes \bs ^* e ^\lambda _r .$$
Using Lemma \ref{TensorD} one estimates the integrand of the l.h.s. of Equation \eqref{EstLem1}, then there exits constant $C_3>0$ such that
\begin{align*}
\Big| \dot \nabla ^\alpha (\hat A \theta _\alpha \psi ^\alpha ) & \Big| ^2 (x, y, z) \\
=& \Big| \sum _r (\dot \nabla ^{\rE _\flat} A \theta _\alpha (u ^\lambda _r |_{\rM _\alpha \times \{ z \}})
(x, y)) \otimes \bs ^* e ^\lambda _r
+ (A \theta _\alpha u ^\lambda _r )\otimes \bs ^* (\nabla ^\rE e ^\lambda _r ) \Big| ^2 \\
\leq & C _3 \sum _r
\Big( \Big| \dot \nabla ^{\rE _\flat} A \theta _\alpha (u ^\lambda _r |_{\rM _\alpha \times \{ z \}})
(x, y) \Big| ^2
+ \Big| (A \theta _\alpha u ^\lambda _r ) \otimes \bs ^* (\nabla ^\rE e ^\lambda _r ) \Big|^2
\Big).
\end{align*}
By integrating, one gets for some constants $C _q$, $q=4,\cdots,10$, that
\begin{align*}
\int _{\rB _\alpha} \int _{\rZ _x} \chi (x, z) \int _{\rZ _x} &
| \dot \nabla ^\alpha \hat A (\theta _\alpha \psi ^\alpha ) | ^2
\mu _x (y) \mu _x (z) \mu _\rB (x) \\
\leq C _4 \sum _\lambda  \int _{\rZ _\lambda} \int _{\rB _\alpha} & \int _{\rZ _x}
\sum _r \Big( \Big| \dot \nabla ^{\rE _\flat} A \theta _\alpha (u ^\lambda _r |_{\rM _\alpha \times \{ z \}})
(x, y) \Big| ^2 \\
&+ \Big| (A \theta _\alpha u ^\lambda _r )\otimes \bs ^* (\nabla ^\rE e ^\lambda _r ) \Big|^2 \Big)
\mu _x (y) \mu _\rB (x) \mu _\rZ (z) \\
\leq \sum _\lambda  \int _{\rZ _\lambda} \int _{\rB _\alpha} & \int _{\rZ _x}
\sum _r \Big( C _5 \| A \| _{\op 1} ^2
\big( \big| \dot \nabla ^{\rE _\flat} \theta _\alpha (u ^\lambda _r |_{\rM _\alpha \times \{ z \}}) (x, y) \big| ^2 \\
&+ \big| \dot \partial ^\rV \theta _\alpha (u ^\lambda _r |_{\rM _\alpha \times \{ z \}}) (x, y) \big| ^2
+ \big| \theta _\alpha (u ^\lambda _r |_{\rM _\alpha \times \{ z \}}) (x, y) \big| ^2 \big) \\
&+ C _6 \| A \| _{\op 0} \big| \theta _\alpha u ^\lambda _r \big|^2 \Big)
\mu _x (y) \mu _\rB (x) \mu _\rZ (z) \\
\leq \sum _\lambda  \int _{\rZ _\lambda} \int _{\rB _\alpha} & \int _{\rZ _x}
J _\alpha (C _7 \| A \| _{\op 1} ^2 + C _8 \| A \| _{\op 0} )
\big( \big| \dot \nabla ^\alpha \theta _\alpha \psi _\alpha \big| ^2 \\
&+ \big| \dot \partial ^\alpha \theta _\alpha \psi _\alpha \big| ^2
+ \big| \dot \partial ^\rZ \theta _\alpha \psi _\alpha \big| ^2 + \big| \theta _\alpha \psi _\alpha \big|^2 \big)
\mu _x (y) \mu _\rB (x) \mu _\rZ (z) \\
\leq \int _\rB \int _{\rZ _x} \chi _\alpha & \int _{\rZ _x}
(C _9 \| A \| _{\op 1} ^2 + C _{10} \| A \| _{\op 0} )
\big( \big| \dot \nabla ^{\hat \rE _\flat} \bx _\alpha ^* ( \theta _\alpha \psi ) \big| ^2 \\
+& \big| \dot \partial ^\bs \bx _\alpha ^* ( \theta _\alpha \psi ) \big| ^2
+ \big| \dot \partial ^\bt \bx _\alpha ^* ( \theta _\alpha \psi ) \big| ^2
+ \big| \bx _\alpha ^* ( \theta _\alpha \psi ) \big|^2 \big)
\mu _x (y) \mu _x (z) \mu _\rB (x).
\end{align*}
Now we use an argument similar to the proof of Lemma \ref{EllReg2}.
Namely, write the integrand as a sum
\begin{align*}
\chi _\alpha \big(  \big| \dot \nabla ^{\hat \rE _\flat} & \bx _\alpha ^* ( \theta _\alpha \psi ) \big| ^2
+ \big| \dot \partial ^\bs \bx _\alpha ^* ( \theta _\alpha \psi ) \big| ^2
+ \big| \dot \partial ^\bt \bx _\alpha ^* ( \theta _\alpha \psi ) \big| ^2
+ \big| \bx _\alpha ^* ( \theta _\alpha \psi ) \big|^2 \big) \\
=& \sum _{g \in S} \chi _\alpha g ^* \chi
\big( \big| \dot \nabla ^{\hat \rE _\flat} \bx _\alpha ^* ( \theta _\alpha \psi ) \big| ^2
+ \big| \dot \partial ^\bs \bx _\alpha ^* ( \theta _\alpha \psi ) \big| ^2
+ \big| \dot \partial ^\bt \bx _\alpha ^* ( \theta _\alpha \psi ) \big| ^2
+ \big| \bx _\alpha ^* ( \theta _\alpha \psi ) \big|^2 \big).
\end{align*}
Then since for all $g$
$$ \int g ^* \chi
\big( \big| \dot \nabla ^{\hat \rE _\flat} \bx _\alpha ^* ( \theta _\alpha \psi ) \big| ^2
+ \big| \dot \partial ^\bs \bx _\alpha ^* ( \theta _\alpha \psi ) \big| ^2
+ \big| \dot \partial ^\bt \bx _\alpha ^* ( \theta _\alpha \psi ) \big| ^2
+ \big| \bx _\alpha ^* ( \theta _\alpha \psi ) \big|^2 \big) = \| \psi \| _{\HS 1},$$
equation \eqref{EstLem1} follows.

Using the same arguments with $\dot \partial ^\alpha $ in place of $\dot \nabla ^\alpha $,
one gets the Equation \eqref{EstLem2}.

As for the last inequality, since $\bt ^{* } \rE |_{\rM _\alpha \times \{ z \}} $
and the connection $(\bx ^{-1} _\alpha ) ^* \nabla ^{\bs ^{*} \rE } $ is trivial along $\exp t Z _0 $,
one can write
\begin{align*}
\nabla ^\alpha _{Z _0 }(\hat A u \otimes \bs ^* e )
=& \frac{d}{d t} \Big|_{t = 0} A u |_{\rM _\alpha \times \{ \exp t Z \} } \otimes \bs ^* e
+ u \otimes \nabla ^{\bs ^{*} \rE '} _{Z _0} \bs ^* e \\
=& A \big( \frac{d}{d t} \Big|_{t = 0} u |_{\rM _\alpha \times \{ \exp t Z \} } \big) \otimes \bs ^* e
+ u \otimes \nabla ^{\bs ^{*} \rE '} _{Z _0} \bs ^* e
= \hat A (\nabla ^\alpha _{Z _0} (u \otimes \bs ^* e )).
\end{align*}
It follows that
$$ \dot \partial ^\rZ \hat A \psi ^\alpha = \hat A (\dot \partial ^\rZ \psi ^\alpha ) ,$$
and from which Equation \eqref{EstLem3} follows.
\end{proof}

Clearly, the arguments leading to Corollary \ref{AEstCor1} can be repeated and we obtain:
\begin{cor}
\label{Main1}
For any smooth bounded operator $\hat A $ and $m = 0, 1 , \cdots $, there exists  $C'_{m,l}> 0 $ such that for any  $\psi \in \Psi ^{- \infty} _\infty (\rM \times _\rB \rM ) ^\rG $, one has
$$ \| \hat A \psi \| _{\HS m}
\leq \big( \sum _{0 \leq l \leq m  } C _{m, l} \| A \|_{\op l} \big) \| \psi \| _{\HS m}.$$
\end{cor}

\begin{nota}
In view of Corollary \ref{Main1}, we shall denote 
$$ \| A \| _{\op' m} := \big( \sum _{0 \leq l \leq m  } C _{m, l} \| A \|_{\op l} \big).$$
\end{nota}
We may assume without loss of generality that $C _{m, l} \geq 2$.
Then one still has
\begin{equation}
\| A _1 A _2 \| _{\op' m} \leq \| A _1 \| _{\op' m} \| A _2 \| _{\op' m} .
\end{equation}

\section{Large time behavior of the heat operator}
In this section we will prove that under the condition of the positivity of the Novikov-Shubin invariant, the heat operator also convergences to the projection operator under the norm $\|\cdot\|_{{\rm HS}\ m}$.
\subsection{The Novikov-Shubin invariant}
Let $\rM \to \rB $ be a fiber bundle with a $\rG$ action,
and $T \rM = \rH \oplus \rV $ be the $\rG$-invariant splitting, as defined in Section \ref{Dfn}.
Recall that we assumed the metric on
$\rH \cong \pi ^{*} T \rB$ is given by pulling back some Riemannian metric on $\rB$.
In other words, $\rV $ is a Riemannian foliation.

Let $\rE \to \rM$ be a flat, contravariant $\rG$-vector bundle,
and $\nabla $ be an invariant flat connection on $\rE$.
Denote $\rE ^\bullet := \wedge ^\bullet \rV' \otimes \rE$.

Since the vertical distribution $\rV$ is integrable,
the deRham differential $d ^{\nabla ^\rE} _\rV$ along $\rV$ is well defined.
Write $\eth _0 := d ^{\nabla ^\rE} _\rV + (d ^{\nabla ^\rE } _\rV )^* , \varDelta := \eth _0^{2} $,
and denote by $e ^{- t \varDelta} $ the heat operator and
$ \varPi _0 $ the orthogonal projection onto $\Ker (\varDelta )$.

The following result is classical.
See, for example, \cite[Proposition 2.8]{B} and \cite[Proposition 3.5]{Heitsch;FoliationHeat}.
\begin{lem}
\label{DuhamelLem}
The heat operator $e ^{- t \varDelta }$ is given by a smooth kernel.
Moreover, for any first order differential operator $A$, one has the Duhamel type formula
\begin{equation}
\label{Duhamel}
[A , e ^{- t \varDelta } ] = - \int _0 ^t e ^{- (t - t') \varDelta } [A , \varDelta ] e ^{- t' \varDelta } d t' .
\end{equation}
\end{lem}


From Lemma \ref{DuhamelLem}, it follows that:
\begin{cor}\label{DuhamelCor}
\cite[Corollary 3.11]{Heitsch;FoliationHeat}
For any $i, j, k$, there exist $C , M > 0$ such that
$$ |(\dot \nabla ^{\rE _\flat}) ^i (\dot \partial ^\bs)^j (\dot \partial ^\bt) ^k e ^{- t \eth _0 ^2}| (x, y, z)
\leq C e ^{- \rM \bd (y, z) ^2}.$$
\end{cor}
Hence $e ^{- t \eth _0 ^2} \in \Psi ^{- \infty} _\infty (\rM \times _\rB \rM , \rE ^\bullet) ^\rG $.

As for $\varPi _0$, one has
\begin{lem}\label{GR}
The kernel of $\varPi _0 $ lies in $\bar \Psi ^{- \infty} _0 (\rM \times _\rB \rM , \rE ^\bullet) ^\rG $.
\end{lem}
\begin{proof}
By \cite[Theorem 2.2]{Gong;CoverTorsion} $\varPi _0$ is also represented by a smooth kernel $\varPi _0 (x, y, z)$.
Moreover by \cite[Theorem 2.2]{Gong;CoverTorsion} and the fact that $\varPi _0 = \varPi _0 ^2 $, one has
$$ \sup _{x \in \rB} \Big\{ \int _{\rZ _x} \chi (x, z) \int _{\rZ _x}
| \varPi _0 (x, y, z) |^2 \mu _x (y) \mu _x (z) \Big\}
= \| \varPi _0 \| _\tau < \infty ,$$
where $\| \cdot \| _\tau $ is the $\tau$-trace norm defined in \cite{Gong;CoverTorsion}
(see also \cite{Schick;NonCptTorsionEst}).

Hence it remains to consider
$\chi _n (x, y, z) \varPi _0 (x, y, z) $,
where $\chi _n \in C ^\infty (\rM \times _\rB \rM) ^\rG$ is a sequence of smooth functions such that
\begin{enumerate}
\item
$0 \leq \chi _n \leq 1$;
\item
$\chi _n$ is increasing and converges point-wise to 1;
\item
$\chi _n (x, y, z) = 0 $ whenever $\bd (y, z) > n r $ for some $r > 0$.
\end{enumerate}
To construct $\chi _n$, let $r > 0 $ to be the infremum of the injective radius of the fibers $\rZ _x$,
and $\phi _1 $
be a non-negative smooth function such that $\phi _1 (t) = 1 $ if $t < \frac{r}{2}$, $\phi _1 (t) = 0 $ if $t > r$.
Then $\chi _1 := \phi _1 \circ \bd (y, z)$ is $\rG$-invariant.
Define
$$ \tilde \chi _n := \chi _1 \star \cdots \star \chi _1 \text{ (convolution by $n$ times).}$$
Note that $\tilde \chi _n (x, y, z) > 0 $ whenever $\bd (y, z) < \frac{n r}{2}$.
Moreover, $\chi _n$ is $\rG$-invariant and $\tilde \chi _n (x, y, z) = 0 $ whenever $\bd (y, z) > n r $.
Since $\tilde \chi _{n + 1}$ is bounded away from $0$ on the support of $\tilde \chi _n$,
clearly one can find smooth functions $\phi _n$ such that $\chi _n := \phi _n \circ \tilde \chi _n$
satisfies conditions (1)-(3).
\end{proof}

Because of Corollary \ref{DuhamelCor} and Lemma \ref{GR}, it makes sense to define:
\begin{dfn}
We say that $\varDelta$ has positive Novikov-Shubin invariant if there exist $\gamma > 0 $ and $C_0>0$ such that
for sufficiently large $t$,
$$ \sup _{x \in \rB} \Big\{ \int _{\rZ _x} \chi (x, z) \int _{\rZ _x}
| (e ^{- t \varDelta } - \varPi _0) (x, y, z) |^2 \mu _x (y) \mu _x (z) \Big\} \leq C _0 t ^{- \gamma }.$$
\end{dfn}

\begin{rem}
The positivity of the Novikov-Shubin invariant is independent of the metrics defining the operator $\varDelta$.
\end{rem}

\begin{rem}
\label{SquareNS}
Since $e ^{- \frac{t}{2} \varDelta } - \varPi _0 $ is non-negative, self adjoint and
$(e ^{- \frac{t}{2} \varDelta } - \varPi _0 ) ^2 = e ^{- t \varDelta } - \varPi _0 $, one has
$$ \sup _{x \in \rB} \Big\{ \int _{\rZ _x} \chi (x, z) \int _{\rZ _x}
| (e ^{- \frac{t}{2} \varDelta } - \varPi _0 ) (x, y, z)|^2 \mu _x (y) \mu _x (z) \Big\}
= \| e ^{- t \varDelta } - \varPi _0 \| _\tau .$$
Hence our definition of having positive Novikov-Shubin invariant is equivalent to that of \cite{Schick;NonCptTorsionEst}.
Our argument here is similar to the proof of \cite[Theorem 7.7]{BMZ}.
\end{rem}

In this paper, we shall always assume $\varDelta $ has positive Novikov-Shubin invariant.
From this assumption, it follows by integration over $\rB$ that there exist constants $\gamma>0$ and $C>0$ such that for $t$ large enough
\begin{equation}
\|  e ^{- t \varDelta } - \varPi _0 \| _{\HS 0} < C t ^{- \gamma }.
\end{equation}

\subsection{Example: The Bismut super-connection}
\begin{dfn}
\label{BismutDfn}
A standard flat Bismut super-connection is an operator of the form
$$ d ^{\nabla ^{\rE}}
:= d ^{\nabla ^{\rE}} _\rV + \nabla ^{\rE ^\bullet _\flat } + \iota _\Theta ,$$
where $\Theta $ is the $\rV$-valued horizontal 2-form defined by
$$ \Theta (X _1 ^\rH , X _2 ^\rH) := - P ^\rV [ X _1 ^\rH , X _2 ^\rH ] ,
\quad \Forall X _1 , X _2 \in \Gamma ^\infty (T \rB), $$
and $\iota _\Theta $ is the contraction with $\Theta$. Note that $P^V$ is not canonical and it depends on the splitting $TM = V\oplus H$.
\end{dfn}

Observe that the adjoint of the Bismut super-connection,
$(d ^{\nabla ^{\rE}})' = (d ^{\nabla ^{\rE}} _\rV ) ^* + (\nabla ^{\rE ^\bullet _\flat } )' - \Lambda _{\Theta ^*} $,
is also flat.
It follows that
$$ (\nabla ^{\rE ^\bullet _\flat })' (d ^{\nabla ^\rE} _\rV ) ^*
+ (d ^{\nabla ^\rE} _\rV ) ^* ( \nabla ^{\rE ^\bullet _\flat } )' = 0. $$
Define
$$ \Omega := \frac{1}{2} \big((\nabla ^{\rE ^\bullet _\flat })' - \nabla ^{\rE ^\bullet _\flat } \big).$$
Observe that $\Omega $ is a tensor (see \cite{Lopez;FoliationHeat} for an explicit formula for $\Omega $).
Moreover one has
$$ \nabla ^{\rE ^\bullet _\flat } (d ^{\nabla ^\rE} _\rV ) ^* +(d ^{\nabla ^\rE} _\rV ) ^* \nabla ^{\rE ^\bullet _\flat }
= 2 \Omega (d ^{\nabla ^\rE} _\rV ) ^* + 2 (d ^{\nabla ^\rE} _\rV ) ^* \Omega .$$
Also, observe that
$(d ^{\nabla ^\rE} _\rV ) + (d ^{\nabla ^\rE} _\rV ) ^* + \nabla ^{\rE ^\bullet _\flat }
+ ((\nabla ^{\rE ^\bullet _\flat } )') ^* $
is an elliptic operator.

\subsection{The regularity result of Alvarez Lopez and Kordyukov}

We first recall that an operator $A$ is called $C ^\infty$-bounded if in normal coordinates the coefficients and their derivatives are uniformly bounded. As in \cite{Lopez;FoliationHeat},
we make the more general assumption that there exists
$C^\infty$-bounded first order differential operator $Q$,
and zero degree operators $R _1, R _2 , R _3 , R _4 $, all $\rG$-invariant,
such that $d _\rV^{\nabla ^{E}} + (d ^{\nabla ^\rE} _\rV ) ^* + Q $ is elliptic, and
\begin{align}
\label{LopezHypo}
Q d _\rV ^{\nabla^{E}}+ d _\rV ^{\nabla^{E}}Q &= R _1 d _\rV ^{\nabla^{E}}+ d _\rV^{\nabla^{E}} R _2, \\ \nonumber
Q (d ^{\nabla^{E}}_\rV)^* + (d ^{\nabla^{E}}_\rV)^* Q &= R _3 (d ^{\nabla^{E}}_\rV)^* + (d ^{\nabla^{E}}_\rV)^* R _4 .
\end{align}
Clearly, in our example,
$\nabla ^{\rE ^\bullet _\flat } + ((\nabla ^{\rE ^\bullet _\flat } )') ^* $ satisfies Equation \eqref{LopezHypo}.

Write $\eth _0 := d ^{\nabla^{E}}_\rV + (d ^ {\nabla^{E}} _\rV)^* , \varDelta := \eth _0 ^2 $,
and denote by $ \varPi _{d _\rV} , \varPi _{d ^* _\rV} $ respectively the orthogonal projections onto
the range of $d ^{\nabla ^\rE} _\rV , (d ^{\nabla ^\rE} _\rV )^* $,
which we shall denote by $\Rg (d _\rV ) , \Rg (d ^* _\rV)$.

In this section, we shall consider the operators
\begin{align*}
B _1 :=& R _1 \varPi _{d _\rV} + R _3 \varPi _{d ^* _\rV}, \\
B _2 :=& \varPi _{d ^* _\rV} R _2 + \varPi _{d _\rV} R _4, \\
B :=& B _2 \varPi _0 + B _1 (\id - \varPi _0 ).
\end{align*}
We recall some elementary formulas regarding these operators from \cite{Lopez;FoliationHeat}:
\begin{lem}
\label{Lopez;2.2}
\cite[Lemma 2.2]{Lopez;FoliationHeat}
One has
\begin{align*}
\nonumber
Q d ^{\nabla^{E}}_\rV + d ^{\nabla^{E}}_\rV Q =& B _1 d ^{\nabla^{E}}_\rV + d^{\nabla^{E}} _\rV B _2, \\
Q (d ^{\nabla^{E}}_\rV)^* + (d ^{\nabla^{E}}_\rV)^* Q =& B _1 (d ^{\nabla^{E}}_\rV )^*+ (d ^{\nabla^{E}}_\rV)^* B _2, \\ \nonumber
[Q , \varDelta ] =& B _1 \varDelta - \varDelta B _2 - \eth _0 (B _1 - B _2 ) \eth _0.
\end{align*}
\end{lem}

One can furthermore estimate the derivatives of $\varPi _0 $.
First, recall that
\begin{lem}
\label{Lopez;2.8}
One has (cf. \cite[Corollary 2.8]{Lopez;FoliationHeat})
$$ [Q + B, \varPi _0 ] = 0 .$$
\end{lem}
\begin{proof}
Here we give a different proof.
From definition we have
$$ B = (\varPi _{d _{V}^*} R _2 + \varPi _d R _4 ) \varPi _0 + R _1 \varPi _{d_{V}} + R _3 \varPi _{d ^*_{V}} ,$$
where we used $\varPi _{d_V} \varPi _0 = \varPi _{d_V ^*} \varPi _0 = 0.$
Hence
$$ B \varPi _0 - \varPi _0 B
= (\varPi _{d ^*_{V}} R _2 + \varPi _{d_{V}} R _4 ) \varPi _0 - \varPi _0 R _1 \varPi _{d_{V}} - \varPi _0 R _3 \varPi _{d ^*_{V}} .$$
For any $s$ one has
$$ \varPi _{d_V} s = \lim _{n \to \infty} d \tilde s _n ,$$
for some sequence $\tilde s _n $ (in some suitable function spaces).
It follows that
\begin{align*}
\varPi _0 R _1 \varPi _{d_V} s
=& \lim _{n \to \infty } \varPi _0 R _1 d \tilde s _1 \\
=& \lim _{n \to \infty } \varPi _0 (Q d^{\nabla^E}_V + d^{\nabla^E}_V Q - d^{\nabla^{E}}_V R _2 ) \tilde s _1
= \varPi _0 Q \varPi _{d_V} s.
\end{align*}
Similarly, one has $\varPi _0 R _3 \varPi _{d_V ^*} = \varPi _0 Q \varPi _{d_V ^*} $ and by considering the adjoint
$\varPi _{d ^*_{V}} R _2 \varPi _0 = \varPi _{d ^*_{V}} Q \varPi _0 $,
and $\varPi _{d ^*_{V}} R _4 \varPi _0 = \varPi _{d ^*_{V}} Q \varPi _0$.
It follows that
\[ [Q + B , \varPi _0 ] = (\id - \varPi _{d _{V}} - \varPi _{d ^*_{V}} ) Q \varPi _0
- \varPi _0 Q (\id - \varPi _{d_{V}} - \varPi _{d ^*_{V}}) = 0 . \qedhere \]
\end{proof}

In other words, regarding
$[Q, \varPi _0 ] $ and $ [B , \varPi _0 ]$
as kernels, one has 
$$\| [Q, \varPi _0 ] \| _{\HS m} = \| [B , \varPi _0 ] \| _{\HS m},$$
provided the right hand side is finite.
Hence, using elliptic regularity and the same arguments as Lemma \ref{GR},
one can prove inductively that
$$\varPi _0 (x, y, z) \in \bar \Psi ^{-\infty} _m (\rM \times _\rB \rM , \rE ^\bullet ), \quad \Forall m.$$

Next, we recall the main result of \cite{Lopez;FoliationHeat}
\begin{lem}
\label{OldLem}
For any $m = 0, 1, \cdots $,
\begin{enumerate}
\item
The heat operator $e ^{- t \varDelta } $, and the operators
$\eth_0 e ^{- t \varDelta }, \varDelta e ^{- t \varDelta } $ map $\cW ^m (\rE )$ to itself as bounded operators.
Moreover, there exist constants $C ^0 _m , C ^1 _m , C ^2 _m > 0 $ such that
\begin{align*}
\| e ^{- t \varDelta } \| _{\op m} \leq & C ^0 _m ,\\
\|\eth_0  e ^{- t \varDelta } \| _{\op m} \leq & t ^{- \frac{1}{2}} C ^1 _m, \\
\| \varDelta e ^{-t \varDelta } \| _{\op m} \leq & t ^{-1} C ^2 _m ,
\end{align*}
for all $t > 0$.
\item
As $t \to \infty $, $e ^{- t \varDelta } $ strongly converges as an operator on $\cW ^m (\rE )$.
Moreover, $(t, s) \mapsto e ^{- t \varDelta } s $ is a continuous map form
$[0, \infty ] \times \cW ^m (\rE ) $ to $\cW ^m (\rE )$.
\item
One has the Hodge decomposition
$$ \cW ^m (\rE ) = \Ker (\varDelta ) + \overline { \Rg (\varDelta )}
= \Ker (\eth _0 ) + \overline {\Rg (\eth _0)}, $$
where the kernel, image and closure are in $\cW ^m (\rE )$.
\end{enumerate}
\end{lem}
Remark that our case is slightly different from that of \cite{Lopez;FoliationHeat},
where $\rM$ is assumed to be compact (but with possibly non-compact fibers).
However, the same arguments clearly apply because our $\rM$ is of bounded geometry.

We recall more results in \cite[Section 2]{Lopez;FoliationHeat}.
\begin{lem}
\label{Lopez;2.4}
\cite[Lemma 2.4]{Lopez;FoliationHeat}
For any $m\geq 0$, there exists constant $C^3_0>0$ such that
$$ \| [Q , e ^{- t \varDelta } ] \| _{\op m} \leq C ^3 _m .$$
\end{lem}
\begin{proof}
Using the third equation of Lemma \ref{Lopez;2.2}, Equation (\ref{Duhamel}) becomes
$$ [Q , e ^{- t \varDelta } ]
= \int _0 ^{t } e ^{- (t - t' ) \varDelta } \eth _0 (B _1 - B _2 ) \eth _0 e ^{- t' \varDelta } d t'
- \int _0 ^{t } e ^{- (t - t' ) \varDelta } (B _1 \varDelta - \varDelta B _2 ) e ^{- t' \varDelta } d t' .$$
Using Lemma \ref{OldLem}, we estimate the first integral
\begin{align*}
\Big\| \int _0 ^{t } e ^{- (t - t' ) \varDelta } \eth _0 (B _1 - B _2 ) \eth _0 e ^{- t' \varDelta } d t' \Big\| _{\op m}
\leq & \| B _1 - B _2 \| _{\op m} (C ^1 _m ) ^2 \int _0 ^t \frac{ d t' }{\sqrt{(t - t' ) t' }} \\
=& \| B _1 - B _2 \| _{\op m} (C ^1 _m ) ^2 \pi.
\end{align*}
As for the second integral,
we split the domain of integration into $[0, \frac{t}{2}] $ and $[\frac{t}{2} , t]$,
and then integrate by part to get
\begin{align*}
\int _0 ^{t } e ^{- (t - t' ) \varDelta } (B _1 \varDelta - \varDelta B _2 ) e ^{- t' \varDelta } d t'
=& \int _0 ^{\frac{t}{2} } e ^{- (t - t' ) \varDelta } \varDelta (- B _1 - B _2 ) e ^{- t' \varDelta } d t' \\
&- \int _{\frac{t}{2}} ^{t } e ^{- (t - t' ) \varDelta } (B _1 - B _2 ) \varDelta e ^{- t' \varDelta } d t' \\
&+ e ^{- (t - t' ) \varDelta } B _1 e ^{- t' \varDelta } \Big |^{\frac{t}{2}} _{t' = 0}
- e ^{- (t - t' ) \varDelta } B _2 e ^{- t' \varDelta } \Big |^{t} _{t' = \frac{t}{2}}.
\end{align*}
Again using Lemma \ref{Lopez;2.2}, its $\| \cdot \| _{\op m}$-norm is bounded by
$$ C ^0 _m C ^1 _m (\| B _1 \| _{\op m} + \| B _2 \| _{\op m} ) \Big( \int _0 ^{\frac{t}{2}} \frac{d t'}{t - t'}
+ \int _{\frac{t}{2}} ^t \frac{d t'}{t'} \Big)
+ C ^0 _m (C ^0 _m + 1 ) (\| B _1 \| _{\op m} + \| B _2 \| _{\op m} ),$$
which is uniformly bounded because
$\int _0 ^{\frac{t}{2}} \frac{d t'}{t - t'} = \int _{\frac{t}{2}} ^t \frac{d t'}{t'} = \log 2 $.
\end{proof}

Lemma \ref{Lopez;2.8} suggests that $[Q + B , e ^{-t \varDelta }] $ converges to zero as $t \to \infty$.
Indeed, we shall prove a stronger result, namely,
$[Q + B , e ^{-t \varDelta } ]$ decay polynomially in the $\| \cdot \| _{\HS m} $-norm for all $m$.

\begin{lem}
\label{Trick1}
Suppose there exist $C_m, \gamma > 0 $ such that $\| e ^{- t \varDelta } - \varPi _0 \| _{\HS m} \leq C _m t ^{- \gamma }$,
then there exist $C' _m, \gamma_m > 0$ such that 
$$ \| [Q + B , e ^{-t \varDelta } ] \| _{\HS m}
= \| [Q + B , e ^{-t \varDelta } - \varPi _0 ] \| _{\HS m} \leq C' _m t ^{- \gamma _m } .$$
\end{lem}
\begin{proof}
We follow the proof of \cite[Lemma 2.6]{Lopez;FoliationHeat}.
By Lemma \ref{Lopez;2.2}, we get
$$ [Q + B , \varDelta ] = (\varDelta (B _1 + B _2 ) + \eth _0 (B _1 - B _2) \eth _0 ) (\id - \varPi _0 ),$$
it follows that
$$ \varPi _0 [Q + B , e ^{- \frac{t}{2} \varDelta }] = [Q + B , e ^{- \frac{t}{2} \varDelta }] \varPi _0 = 0 .$$
Write
\begin{align*}
 [Q + B , e ^{-t \varDelta } ]
=& [Q + B , e ^{ - \frac{t}{2} \varDelta } ] e ^{- \frac{t}{2} \varDelta }
+ e ^{- \frac{t}{2} \varDelta } [Q + B , e ^{- \frac{t}{2} \varDelta } ] \\
=& [Q + B , e ^{ - \frac{t}{2} \varDelta } ] ( e ^{- \frac{t}{2} \varDelta } - \varPi _0 )
+ (e ^{- \frac{t}{2} \varDelta } - \varPi _0 ) [Q + B , e ^{- \frac{t}{2} \varDelta } ].
\end{align*}
Taking $\| \cdot \| _{\HS m} $ and using
Corollary \ref{Main1} and Lemma \ref{Lopez;2.4}, the claim follows.
\end{proof}

\begin{thm}
\label{Main2}
Suppose $\| e ^{- t \varDelta } - \varPi _0 \| _{\HS 0} \leq C _0 t ^{- \gamma }$ for some $\gamma > 0 , C _0 > 0$.
Then for any $m$, there exists $C'' _m > 0 $ such that
$$\| e ^{- t \varDelta } - \varPi _0 \| _{\HS m} \leq C'' _m t ^{- \gamma } , \quad \Forall t > 1 .$$
\end{thm}
\begin{proof}
We prove the theorem by induction. The case $m = 0 $ is given.
Suppose that for some $m$, $\| e ^{- t \varDelta } - \varPi _0 \| _{\HS m} \leq C _m t ^{- \gamma } .$
Consider $\| e ^{- t \varDelta } - \varPi _0 \| _{\HS m +1}$.

Since $Q $ is a first order differential operator,
for any kernel $\psi \in \Psi ^{- \infty } _\infty (\rM \times _\rB \rM , \rE ^\bullet ) ^\rG$,
$[Q , \psi ] $ is also a kernel lying in $\Psi ^{- \infty } _\infty (\rM \times _\rB \rM , \rE ^\bullet ) ^\rG$,
that is in particular given by a composition of the covariant derivatives
$\dot \nabla ^{\hat \rE _\flat} , \dot \partial ^\bs , \dot \partial ^\bt $ and some tensors acting on $\psi $.
Since $\| \psi \| _{\HS m} $ is by definition the $\| \cdot \| _{\HS 0} $ norm of the $m$-th derivatives of $\psi$,
elliptic regularity (Lemma \ref{EllReg2}) implies
$$ \| \psi \| _{\HS m+1}
\leq \tilde C _m (\| \psi \|_{\HS m} + \| \eth_0 \psi \|_{\HS m} + \| \psi \eth _0 \| _{\HS m}+\| [Q , \psi ] \| _{\HS m}),$$
for some constant $\tilde C _m > 0$.
Put $\psi = e ^{- t \varDelta } - \varPi _0 $.
The theorem then follows from the estimates
\begin{align*}
\| \eth_0 (e ^{- t \varDelta } - \varPi _0 ) \| _{\HS m}
= & \| (e ^{- t \varDelta } - \varPi _0 ) \eth _0 \| _{\HS m} \\
\leq & \big( \sum _{0 \leq l \leq m } C ' _{m, l} \| \eth _0 ( e ^{- \frac{t}{2} \varDelta } - \varPi _0 ) \| _{\op l} \big)
\| e ^{- \frac{t}{2} \varDelta } - \varPi _0 \| _{\HS m} \\
\leq & \big( \sum _{0 \leq l \leq m } C ' _{m, l} C ^1 _l \big( \frac{t}{2} \big) ^{- \frac{1}{2} } \big)
C _m \big( \frac{t}{2} \big) ^{- \gamma }, \\
\| [Q , e ^{- t \varDelta } - \varPi _0 ] \| _{\HS m}
\leq & \| [Q + B , e ^{- t \varDelta } - \varPi _0 ] \| _{\HS m}
+ \| [B , e ^{- t \varDelta } - \varPi _0 ] \| _{\HS m} \\
\leq & C ' _m t ^{- \gamma }
+ 2 \big( \sum _{0 \leq l \leq m  } C ' _{m, l} \| B \|_{\op l} \big) C _m t ^{- \gamma }.
\end{align*}
Note that we used Lemma \ref{Trick1} for the last inequality.
\end{proof}

\section{Sobolev convergence}
In this section we will use the method of \cite{Schick;NonCptTorsionEst} to prove that under the condition of positivity of the Novikov-Shubin invariant the $L^2$-analytic torsion form is a smooth form.

Let $\nabla ^\rE $ be a flat connection on $\rE$.
Define the number operators on $\wedge ^\bullet \rH' \otimes \wedge ^\bullet \rV' \otimes \rE $ by
$$N _\Omega |_{\wedge ^q \rH' \otimes \wedge ^{q'} \rV' \otimes \rE } := q ,
\quad N |_{\wedge ^q \rH' \otimes \wedge ^{q'} \rV' \otimes \rE } := q' .$$
In this section,
we consider the rescaled Bismut super-connection \cite[Chapter 9.1]{BGV;Book}
$$ \eth (t) := \frac{t ^{\frac{1}{2}}}{2} t ^{- \frac{N_\Omega }{2} } (d + d ^*) t ^{ \frac{N _\Omega}{2} }
= \frac{1}{2} \big( t ^{\frac{1}{2}} (d _\rV + d _\rV ^*)
+ (\nabla ^{\rE _\flat} + (\nabla ^{\rE _\flat }) ')
+ t ^{- \frac{1}{2}} (- \Lambda _{\Theta ^*} + \iota _\Theta ) \big).$$
Denote
$$ D _0 := - \frac{1}{2}(d _\rV - d ^* _\rV ),
\quad \Omega _t := - \frac{1}{2} ((\nabla ^{\rE _\flat} - (\nabla ^{\rE _\flat }) ')
- \frac{t ^{-\frac{1}{2}}}{2} (- \Lambda _{\Theta ^*} - \iota _\Theta ),
\quad D (t) := t ^{\frac{1}{2}} D _0 + \Omega _t .$$
The curvature of $\eth (t)$ can be expanded in the form:
$$ \eth (t) ^2 = - D (t) ^2
= t \varDelta + t ^{\frac{1}{2}} \Omega _t D _0 + t ^{\frac{1}{2}} D _0 \Omega _t + \Omega _t ^2 .$$
Hence as a consequence of Duhamel's expansion (cf. \cite{BGV;Book}), we have
\begin{align}
\nonumber
e ^{ - \eth (t) ^2 } = e ^{ D (t) ^2 } = e ^{- t \varDelta }
+ \sum _{n = 1} ^{\dim \rB} \int _{(r _0 , \cdots , r _k ) \in \Sigma ^n} &
e ^ {- r _0 t \varDelta }
(t ^{\frac{1}{2}} \Omega _t D _0 + t ^{\frac{1}{2}} D _0 \Omega _t + \Omega _t ^2 )
e ^ {- r _1 t \varDelta }
\\ \nonumber
& \cdots (t ^{\frac{1}{2}} \Omega _t D _0 + t ^{\frac{1}{2}} D _0 \Omega _t + \Omega _t ^2 )
e ^{- r _n t \varDelta } d \Sigma ^n,
\end{align}
where $\Sigma ^n := \{(r _0 , r _1 \cdots , r _n ) \in [0, 1] ^{n + 1} : r _0 + \cdots + r _n = 1 \}$.

\subsection{The large time estimate of the rescaled heat operator}
In this section,
we follow \cite[Section 4]{Schick;NonCptTorsionEst} to estimate the Hilbert-Schmit norms of $e ^{ - \eth (t) ^2 }$ (see Theorem \ref{Main3} below).

Let $\gamma ' := 1 - (1 + \frac{2 \gamma }{\dim \rB + 2+2\gamma } ) ^{-1} $,
$\bar r (t) := t ^{- \gamma '}$. Fix $\bar t $ such that $\bar r ( \bar t) < (\dim \rB + 1 ) ^{-1} $.
One has the following counterparts of \cite[Lemma 4.2]{Schick;NonCptTorsionEst}:
\begin{lem}
For $c = 0, 1, 2$, there exists a constant $C_{m}$ such that 
$$ \| (\sqrt {t} \eth _0 ) ^{\frac{c}{2}} e ^{r t (D _0) ^2 } \| _{\op ' m} \leq C _m r ^{- \frac{c}{2}} ,
\text{for any $t > \bar t , 0 < r < 1$ (by Lemma \ref{OldLem})} ;$$
And for any $t > \bar t , \bar r (t) < r < 1$, 
\begin{align*}
\| e ^{r t (D _0) ^2 } \| _{\HS m} \leq & C _m (r t) ^{- \gamma} ,& \text{(by Theorem \ref{Main2})} \\
\| (\sqrt {t} \eth _0 ) ^{\frac{c}{2}} e ^{r t (D _0) ^2 } \| _{\HS m}
\leq & C _m r ^{- \frac{c}{2}} (r t) ^{- \gamma}, \text{ if $c = 1, 2$}. & \text{(by Corollary \ref{Main1})}
\end{align*}
\end{lem}
We furthermore observe that the arguments leading to the main result \cite[Theorem 4.1]{Schick;NonCptTorsionEst} still hold if one replaces the operator and 
$\|\cdot\|_{\tau}$ norm respectively by $\|\cdot\|_{\op' m}$ and $\|\cdot\|_{\HS m}$ for any $m$.

The arguments in \cite[Section 4]{Schick;NonCptTorsionEst} are elementary, so we shall only recall some key steps. 

First, one splits the domain of integration
$\Sigma ^n = \bigcup _{I \neq \{ 0, \cdots , n \}} \Sigma ^n _{\bar r (t), I} $,
where
$$ \Sigma ^n _{\bar r(t) , I} := \{ (r _0 , \cdots , r _n ) : r _i \leq \bar r (t), \Forall i \in I ,r_{j}\geq \bar{r}(t),\Forall j\notin I\}.$$
Define
\begin{equation}
\label{KDfn}
K (t, n, I, c _0 , \cdots c _n ; a _1 , \cdots a _n )
:= \int _{\Sigma ^n _{\bar r(t) , I}}
(t ^{\frac{1}{2}} D _0 ) ^{c _0} e ^ {- r _0 t \varDelta }
\prod _{i=1} ^n ( \Theta _t ^{a _i } (t ^{\frac{1}{2}} D _0 )^{c _i} e ^ {- r _i t \varDelta }) d \Sigma ^n ,
\end{equation}
for $c _i = 0, 1, 2, a _j = 1, 2 $.
Then one has
$$ e ^{- \eth (t) ^2}
= e ^{D (t) ^2} = \sum K (t, n, I, c _0 , \cdots c _n ; a _1 , \cdots a _n ) $$
by grouping terms involving $D _0$ together.

We shall consider the kernels 
$ K (t, n, I, c _0 , \cdots c _n ; a _1 , \cdots a _n ) (x, y, z)$
of the terms in the summation above.
Consider the special case when $c _i = 0, 1$.
One has the analogue of \cite[Proposition 4.6]{Schick;NonCptTorsionEst}:
\begin{lem}
\label{Schick4.6}
There exists $\varepsilon>0$ such that as $t \to \infty $, 
\begin{align*}
K (t , n, I, c _0 , \cdots c _n &, a _1, \cdots , a _n ) (x, y, z) \\
=&
\left\{
\begin{array}{ll}
(\frac{1}{n !} \varPi _0 \Omega ^{a _1} \varPi _0 \cdots \varPi _0) (x, y, z)+ O (t ^{- \varepsilon })
& \text{ if } I = \emptyset , c _0 , \cdots , c _n = 0 \\
O (t ^{- \varepsilon }) & \text{ otherwise}
\end{array}
\right.
\end{align*}
in the $\| \cdot \| _{\HS m}$-norm.
\end{lem}
\begin{proof}
We first consider the case $I=\emptyset$. Suppose furthermore $c_{q}=1$ for some $q$.
By Corollary \ref{Main1}, The $\| \cdot \| _{\HS m}$-norm of the integrand on the r.h.s. of \eqref{KDfn} is bounded by 
$$\|(t^{{1\over 2}}D_{0})^{c_{0}}e^{-r_{0}t\Delta}\|_{\op' m}\cdots \|\Omega_{t}^{a_{q}}\|_{\op' m}\|(t^{1\over 2}D_{0})e^{-r_{q}t\Delta}\|_{\HS m}\cdots \|(t^{1\over 2}D_{0})^{c_{n}}e^{-r_{n}t\Delta}\|_{\op' m}$$
$$\leq C_{m}'r_0^{-{{c_0}\over 2}}\cdots r_q^{-{{c_q}\over 2}}(r_q t)^{-\gamma}\cdots r_n^{-{{c_n}\over 2}}$$
$$\leq C_{m}'\bar{r}(t)^{-{n\over 2}-\gamma}t^{-\gamma}.$$
Integrating, we have the estimate
$$\|K(t,n,c_0,\cdots,c_n;a_1,\cdots,a_n) (x, y, z)\|_{\HS m} 
\leq C_{m}'t^{-\gamma+\gamma'({{n\over 2}}+\gamma)}\int d\Sigma^n,$$
which is $O(t^{-\varepsilon})$ with $\varepsilon=\gamma(1-{{{\rm dim}B+2\gamma}\over{{\rm dim}B}+2+2\gamma})$.

Next, suppose $I=\emptyset$ and $c_i=0$ for all $i$. Write $e^{-r_0 t\Delta-\Pi_0}+\Pi_0$ and split the integrand 
$$(e^{-r_0 t\Delta}\Omega_{t}^{a_1}e^{-r_1 t\Delta}\cdots e^{-r_n t\Delta} ) (x, y, z)$$
into $2^{n+1}$ terms. If any term contains a $e^{-r_i t\Delta}-\Pi_0$ factor, similar arguments as above shows that it is $O(t^{-\gamma})$. Hence the only term that dose not converge to $0$ is 
$$(\varPi_0 \Omega^{a_1} \varPi_0\cdots \varPi_0 )(x, y, z).$$
Since the volume of $\Sigma^{n}_{\bar{r}(t),I}$ converges to $1\over{n!}$ as $t\to \infty$, the claim follows.

It remains to consider the case when $I$ is non-empty. Write $I=\{i_1,\cdots,i_s\}$, $\{0,\cdots,n\}\setminus I=:\{k_1,\cdots,k_{s'}\}\neq \emptyset$. For $t$ sufficiently large $I\neq \{0,\cdots,n\}$. Suppose $c_q=1$ for some $q\notin I$. Then we take $\|\cdot\|_{\HS m}$-norm for $(t^{1\over 2}D_0)e^{-r_q t\Delta}$ term, and estimate 
\begin{align*}
\|K(t,n,c_0,\cdots,c_n;a_1,\cdots,a_n) (x, y, z)\|_{\HS m} & \\
\leq \int_{0}^{\bar{r}(t)}\cdots \int_{0}^{\bar{r}(t)}
\Big(\int_{\{(r_{k_1},\cdots,r_{k_s'}):(r_0,\cdots,r_n)\in\Sigma^{n}_{\bar{r}(t),I}\}} & C'_m r_0^{-{{c_0}\over 2}}\cdots r_q ^{-{{c_q}\over 2}}(r_q t)^{-\gamma}\cdots r_n^{-{{c_n}\over 2}} \\
& d(r_{k_1}\cdots r_{k_{s'}})\Big) dr_{i_1}\cdots dr_{i_s}.
\end{align*}
As in the $I=\emptyset$ case, the integral over $\{(r_{k_1}),\cdots,r_{k_{s'}}:(r_0,\cdots,r_n)\in \Sigma^{n}_{\bar{r}(t),I}\}$ is $O(t^{-\varepsilon})$; while $\int_{0}^{\bar{r}(t)}r_i^{{c_i}\over 2}dr_i=O(t^{-\gamma'(1-{{c_i}\over 2})})$. Again the claim is verified.

Finally if $c_i=0$ for all $i\in I$, then
\begin{align*}
\|K(t,n,c_0,\cdots,c_n;a_1,\cdots,a_n) (x, y, z)\|_{\HS m} & \\
\leq \int_{0}^{\bar{r}(t)}\cdots \int_{0}^{\bar{r}(t)}\Big(\int_{\{(r_{k_1},\cdots,r_{k_s'}):(r_0,\cdots,r_n)\in\Sigma^{n}_{\bar{r}(t),I}\}} & C''_m r_0^{-{{c_{i_1}}\over 2}}\cdots r_{n}^{-{{c_{i_s}}\over 2}} \\
& d(r_{k_1}\cdots r_{k_{s'}}) \Big) dr_{i_1}\cdots dr_{i_s} \\
&= O(t^{-\gamma'(1-{{c_i}\over 2})}). \qedhere
\end{align*}
\end{proof}

One then turns to the case for some $i$, $c _i = 2$.
If $I$ and $J$ are disjoint subsets of $\{ 0, \cdots , n \}$ with $I = \{ i _1, \cdots , i _r \}$,
and $\{ 0, \cdots , n \} \setminus (I \bigcup J) =: \{ k _0, \cdots , k _q \} \neq \emptyset $,
denote by
$$ \Sigma ^n _{\bar r (t), I, J}
:= \{ (r _0 , \cdots , r _n) \in \Sigma ^n _{\bar r (t), I} : r _j = \bar r (t), \text{ whenever } j \in J \} ,$$
and define
\begin{align*}
K (t, n, I, J, c _0 , \cdots c _n &; a _1 , \cdots a _n ) \\
:= \int _0 ^{\bar r (t)} \cdots \int _0 ^{\bar r (t)} &
\int _{ \{ (r _{k _0}, \cdots r _{k _q}) : (r _0 , \cdots , r _n) \in \Sigma ^n _{\bar r (t), I} \}} \\
(t ^{\frac{1}{2}} D _0 ) ^{c _0} & e ^ {- r _0 t \varDelta }
\prod _{i=1} ^n ( \Theta _t ^{a _i } (t ^{\frac{1}{2}} D _0 )^{c _i} e ^ {- r _i t \varDelta })
\Big| _{\Sigma ^n _{\bar r (t), I, J}}
d ^q (r _{k _0}, \cdots r _{k _q}) d r _1 \cdots d r _r .
\end{align*}
Using integration by parts, one gets \cite[Equation (4.17)]{Schick;NonCptTorsionEst},
\begin{align}
\label{IntPart}
\nonumber
K(t, &n, I \cup \{ i_p \} , J; \cdots, 2, \cdots , c _{k _0}, \cdots ; \cdots , a _{i _p}, a _{i _{p+1}}, \cdots ) \\
=& \left\{
\begin{array}{ll}
K(t, n, I, J \cup \{ i _p \} ; \cdots, 0, \cdots , c _{k _0}, \cdots ; \cdots , a _{i_p} , a _{i _{p+1}}, \cdots ) & \\
- K(t, n - 1, I, J; \cdots, \cdots, c _{k _0}, \cdots ; \cdots , a _{i _0} + a _{i _{p+1}}, \cdots )
& q > 0 ,\\ 
+ K(t, n, I \cup \{ i _p \} , J \cup \{ k _0 \} ; \cdots , 0, \cdots , c _{k _0}, \cdots ;
\cdots , a _{i _p} , a _{i _{p +1}}, \cdots ) \\
+ K(t, n, I \cup \{ i _p \} , J; \cdots , 0, \cdots, c _{k _0} + 2, \cdots ;
\cdots, a _{i _p} , a _{i _{p +1}}, \cdots ) \\
K(t, n, I, J \cup \{ i _p \} ; \cdots, 0, \cdots , c _{k _0}, \cdots ; \cdots , a _{i_p} , a _{i _{p+1}}, \cdots ) & \\
- K(t, n - 1, I, J; \cdots, \cdots, c _{k _0}, \cdots ; \cdots , a _{i _0} + a _{i _{p+1}}, \cdots )
& q = 0. \\ 
+ K(t, n, I \cup \{ i _p \} , J; \cdots , 0, \cdots, c _{k _0} + 2, \cdots ;
\cdots, a _{i _p} , a _{i _{p +1}}, \cdots )
\end{array}
\right.
\end{align}
We remark that the proof of \cite[Equation (4.17)]{Schick;NonCptTorsionEst} does not involve any norm,
therefore we omit the details here.

Using Equation \eqref{IntPart} repeatedly, one eliminates all terms with $c _i = 2$.

On the other hand one has the following straightforward generalization of Lemma \ref{Schick4.6} (compare with \cite[Proposition 4.7]{Schick;NonCptTorsionEst}):
\begin{lem}
Suppose $c _i = 0, 1$. As $t \to \infty $,
\begin{align*}
K (t , n, I, J &, c _0, \cdots c _n ; a _1 , \cdots , a _n ) (x, y, z) \\
=&
\left\{
\begin{array}{ll}
(\frac{1}{(n - |J|)!} \varPi _0 \Omega ^{a _1} \varPi _0 \cdots \varPi _0 )(x, y, z) + O (t ^{- \gamma '})
& \text{ if } I = \emptyset , c _0 , \cdots , c _n = 0 \\
O (t ^{- \gamma' }) & \text{ otherwise,}
\end{array}
\right.
\end{align*}
for some $\gamma ' > 0$, in the $\| \cdot \| _{\HS m}$-norm.
\end{lem}

Thus the term $ K (t, n, I, c _0 , \cdots c _n ; a _1 , \cdots a _n ) $ converges to $0$ unless
$$ c _i = 0 \text{ whenever } i \in I , \quad c _i = 2 \text { whenever } i \not \in I. $$
Then one follows exactly as \cite[Section 4.5]{Schick;NonCptTorsionEst} to compute the limit,
and concludes with the following analogue of \cite[Theorem 4.1]{Schick;NonCptTorsionEst}:
\begin{thm}
\label{Main3}
For $k = 0, 1, 2$ and any $m \in \bbN $,
$$ \lim _{t \to \infty} D (t) ^k e ^{- \eth (t) ^2 } (x, y, z)
= \varPi _0 (\Omega \varPi _0 ) ^k
e ^{ (\Omega \varPi _0 ) ^2 } (x, y, z)$$
in the $\| \cdot \| _{\HS m} $-norm,
where $\Omega := - \frac{\nabla ^{\rE _\flat} - (\nabla ^{\rE _\flat }) ^* }{2}$.
Moreover, there exits $\varepsilon'>0$ such that as $t \to \infty $,
$$ \big\| (D (t) ^k e ^{- \eth (t) ^2 }
- \varPi _0 (\Omega \varPi _0 ) ^k
e ^{ (\Omega \varPi _0 ) ^2 } )(x, y, z) \big\| _{\HS m}
= O (t ^{- \varepsilon '} ).$$
\end{thm}

\subsection{Application: the $L ^2 $-analytic torsion form}
Our main application of Theorem \ref{Main3}
is in establishing the smoothness and transgression formula of the $L ^2 $-analytic torsion form.
Here, we briefly recall the definitions.

On $\wedge ^\bullet T ^* \rM \otimes \rE
\cong \wedge ^\bullet \rH' \otimes \wedge ^\bullet \rV' \otimes \rE $,
define $N_\Omega , N $ to be the number operators of
$\wedge ^\bullet \rH' \cong \pi ^{-1} (\wedge ^\bullet T^* \rB )$ and $\wedge ^\bullet \rV'$ respectively.

Define
$$ F ^\wedge (t)
:= (2 \pi \sqrt{-1}) ^{- \frac{N _\Omega }{2}}
\str _\Psi (2 ^{-1} N (1 + 2 D (t) ^2 ) e ^{- \eth (t) ^2 }).$$
Then under the positivity of the Novikov-Shubin invariant, we have the following well-defined $L^{2}$-analytic torsion form.
\begin{dfn}\label{TorsionDfn}(\cite{Schick;NonCptTorsionEst})
$$ \tau := \int _0 ^\infty \Big\{ - F ^\wedge (t)
+ \frac{\str _\Psi (N \varPi _0 )}{2}
+ \big(\frac{\dim (\rZ) \rk (\rE) \str _\Psi (\varPi _0 )}{4} - \frac{\str _\Psi (N \varPi _0 )}{2} \big)
(1 - 2 t) e ^{-t} \Big\} \frac {d t}{t} .$$
\end{dfn}

In \cite{Schick;NonCptTorsionEst}, it is only shown that the form $\tau$ is continuous. Next we will show that indeed the form $\tau$ is smooth.
\begin{thm}
\label{MainThm}
The form $\tau$ is smooth, i.e. $\tau \in \Gamma ^\infty (\wedge ^\bullet T ^* \rB)$.
\end{thm}
\begin{proof}
Using \cite[Proposition 9.24]{BGV;Book}, the derivatives of the $t$-integrand are bounded as $t \to 0$.
It follows that its integral over $[0, 1 ]$ is smooth.

We turn to study the large time behavior.
Consider $\str ( 2 ^{-1} N (e ^{- \eth (t) ^2} - \varPi _0 ))$.
Using the semi-group property, we can write
$$ e ^{- \eth (t) ^2} = 2 ^{- \frac{N _\Omega}{2}} e ^{- \eth (\frac{t}{2}) ^2}
e ^{- \eth (\frac{t}{2}) ^2} 2 ^{\frac{N _ \Omega}{2}} .$$
Also, since
$\str (N \varPi _0 (\Omega \varPi _0 ) ^{2 j} )
= \str ([N \varPi _0 (\Omega \varPi _0 ) , \varPi _0 (\Omega \varPi _0 ) ^{2 j - 1} ]) = 0$
for any $j \geq 1$ one has
$$ \str (N \varPi _0 )
= \str (N\varPi _0 e ^{(\Omega \varPi _0 ) ^2 })
= 2 ^{- \frac{N _\Omega}{2}}
\str (N \varPi _0 e ^{(\Omega \varPi _0 ) ^2 } \varPi _0 e ^{(\Omega \varPi _0 ) ^2 }) .$$
Therefore
\begin{align*}
\str ( 2 ^{-1} N (e ^{- \eth (t) ^2} - \varPi _0 ))
=& 2 ^{- \frac{N _\Omega}{2}}  \str ( 2 ^{-1} N ( e ^{- \eth (\frac{t}{2}) ^2} e ^{- \eth (\frac{t}{2}) ^2}
- \varPi _0 e ^{(\Omega \varPi _0 ) ^2 } \varPi _0 e ^{(\Omega \varPi _0 ) ^2 })) \\
=& 2 ^{- \frac{N_\Omega}{2}} \str \big( 2 ^{-1} Ne ^{- \eth (\frac{t}{2}) ^2}
( e ^{- \eth (\frac{t}{2}) ^2} - \varPi _0 e ^{(\Omega \varPi _0 ) ^2}) \big) \\
&+ 2 ^{- \frac{N _\Omega}{2}} \str \big( 2 ^{-1} N
( e ^{- \eth (\frac{t}{2}) ^2} - \varPi _0 e ^{(\Omega \varPi _0 ) ^2 }) \varPi _0 e ^{(\Omega \varPi _0 ) ^2 } \big).
\end{align*}

Now consider the $L ^2 (\rB)$-norm of
$ \str _\Psi \big( 2 ^{-1} N e ^{- \eth (\frac{t}{2}) ^2}
( e ^{- \eth (\frac{t}{2}) ^2} - \varPi _0 e ^{(\Omega \varPi _0 ) ^2}) \big)$.
To shorten notations, denote
$G :=  2 ^{-1} N e ^{- \eth (\frac{t}{2}) ^2}
( e ^{- \eth (\frac{t}{2}) ^2} - \varPi _0 e ^{(\Omega \varPi _0 ) ^2}) $.
Writing $G$ as a convolution product, then there exists constant $C_0>0$ such that
\begin{align*}
\int _{\rB} \Big| \int _{\rZ _x} & \chi (x, z) \str (G (x, z, z)) \mu _x (z) \Big|^2 \mu _\rB (x) \\
= \int _\rB & \Big| \int _{\rZ _x} \chi \str \Big( \frac{N }{2} \int _{y \in \rZ _x}
e ^{- \eth (\frac{t}{2}) ^2} (x, z, y)
( e ^{- \eth (\frac{t}{2}) ^2} - \varPi _0 e ^{(\Omega \varPi _0 ) ^2}) (x, y, z) \mu _x (y) \Big) \mu _x (z)
\Big| ^2 \mu _\rB (x)\\
\leq C _0 & \int _\rB \Big( \int _{\rZ _x} \chi \int _{y \in \rZ _x}
\big| e ^{- \eth (\frac{t}{2}) ^2} \big| (x, z, y)
\big| e ^{- \eth (\frac{t}{2}) ^2} - \varPi _0 e ^{(\Omega \varPi _0 ) ^2} \big| (x, y, z) \mu _x (y) \mu _x (z)
\Big) ^2 \mu _\rB (x)\\
\leq & C _0 \| e ^{- \eth (\frac{t}{2}) ^2} \| ^2 _{\HS 0 }
\| e ^{- \eth (\frac{t}{2}) ^2} - \varPi _0 e ^{(\Omega \varPi _0 ) ^2} \| ^2 _{\HS 0},
\end{align*}
where we used the Cauchy-Schwarz inequality three times. Since
$\| e ^{- \eth (\frac{t}{2}) ^2} \| _{\HS 0 } $ is bounded for $t$ large (by triangle inequality),
the expression above is $O (t ^{- \gamma '})$.

We turn to estimate its derivatives.
For any vector field $X$ on $\rB$,
\begin{align*}
\nabla ^{T \rB} _X \str _\Psi (G)
=& \int (L _{X ^\rH } \chi (x, z)) \str (G (x, z, z)) \mu _x (z) \\
&+ \int \chi (x, z) ( L ^{\nabla ^{\pi ^{-1} T \rB }} _{X ^\rH } \str (G (x, z, z))) \mu _x (z) \\
&+ \int \chi (x, z) \str (G (x, z, z)) (L _{X ^\rH }\mu _x (z)) .
\end{align*}
Differentiating under the integral sign is valid because we knew a-priori that the integrands are all $L ^1 $.
Since $L _{X ^\rH }\mu _x (z) $ equals $\mu _x (z) $ multiplied by some bounded functions,
it follows that the last term $\int \chi (x, z) \str (G (x, z, z)) (L _{X ^\rH }\mu _x (z)) $ is $O (t ^{- \gamma ' })$.

For the first term, we write
$ L _{X ^\rH } \chi (x, z) = \sum _{g \in \rG } (g ^* \chi) (x, z) (L _{X ^\rH } \chi ) (x, z).$
The sum is finite because $L _{X ^\rH } \chi $ is compactly supported.
By $\rG$-invariance,
$$ \int  (g ^* \chi ) (x, z) \str (G ( x, z, z)) \mu _x (z)
= \int \chi (x, z) \str (G (x, z, z)) \mu _x (z) .$$
Since $(L _{X ^\rH } \chi ) (x, z)$ is bounded, it follows that
$ \int (L _{X ^\rH } \chi ) (x, z) \str (G (x, z, z)) \mu _x (z) $ is also $O (t ^{- \gamma ' })$.

As for the second term,
we differentiate under the integral sign, then use the Leibniz rule to get that there exists constant $C_1>0$ such that
\begin{align*}
| L ^{\nabla ^{\pi ^{-1} T \rB }} _{X ^\rH } & \str (G (x, z, z))| \\
\leq & C _1 \Big( \int _{\rZ _x}
\big| L ^{\nabla ^{\wedge ^\bullet \rH' \otimes \wedge ^\bullet \rV' \otimes \hat \rE }} _{X ^\rH }
e ^{- \eth (\frac{t}{2}) ^2} (x, z, y) \big|
\big| e ^{- \eth (\frac{t}{2}) ^2} - \varPi _0 e ^{(\Omega \varPi _0 ) ^2} (x, y, z) \big| \mu _x (y) \\
&+ \int _{\rZ _x}
\big| e ^{- \eth (\frac{t}{2}) ^2} (x, z, y) \big|
\big| L ^{\nabla ^{\wedge ^\bullet \rH' \otimes \wedge ^\bullet \rV' \otimes \hat \rE }} _{X ^\rH }
(e ^{- \eth (\frac{t}{2}) ^2} - \varPi _0 e ^{(\Omega \varPi _0 ) ^2}) (x, y, z) \big| \mu _x (y) \\
&+ \int _{\rZ _x}
\big| e ^{- \eth (\frac{t}{2}) ^2} (x, z, y) \big|
\big| e ^{- \eth (\frac{t}{2}) ^2} - \varPi _0 e ^{(\Omega \varPi _0 ) ^2} (x, y, z) \big|
\sup | L _{X ^\rH } \mu | \mu _x (y) \Big), \\
\int _\rB \Big| \int _{\rZ _x} & \chi (x, z)
\big( L ^{\nabla ^{\pi ^{-1} T \rB }} _{X ^\rH } \str (G (x, z, z)) \big) \mu _x (z) \Big|^2 \mu _\rB (x) \\
\leq & C _1 \big( \| e ^{- \eth (\frac{t}{2}) ^2} \|^2 _{\HS 1 }
\| e ^{- \eth (\frac{t}{2}) ^2} - \varPi _0 e ^{(\Omega \varPi _0 ) ^2} \|^2 _{\HS 0} \\
&+ \| e ^{- \eth (\frac{t}{2}) ^2} \|^2 _{\HS 0 }
\| e ^{- \eth (\frac{t}{2}) ^2} - \varPi _0 e ^{(\Omega \varPi _0 ) ^2} \|^2 _{\HS 1} \\
&+ \sup | L _{X ^\rH } \mu | \| e ^{- \eth (\frac{t}{2}) ^2} \|^2 _{\HS 1 }
\| e ^{- \eth (\frac{t}{2}) ^2} - \varPi _0 e ^{(\Omega \varPi _0 ) ^2} \|^2 _{\HS 0} \big) \\
=& O (t ^{- \gamma ' }).
\end{align*}
Clearly the above arguments can be repeated and one concludes that all Sobolev norms of
$\str _\Psi (G) $ are $O (t ^{- \gamma '})$.

By exactly the same arguments, we have as $t \to \infty $,
$$ \str _\Psi \big( 2 ^{-1} N
( e ^{- \eth (\frac{t}{2}) ^2} - \varPi _0 e ^{(\Omega \varPi _0 ) ^2 }) \varPi _0 e ^{(\Omega \varPi _0 ) ^2 } \big)
= O (t ^{- \gamma ' }) ,$$
in all Sobolev norms.

As for $\str _\Psi ( 2 ^{-1} N ( D (t) ^2 e ^{- \eth (t) ^2} )) $,
one has $D (t) ^2 = 2 (2 ^{- \frac{ N _\Omega }{2}} D (\frac{t}{2}) ^2 2 ^{\frac{N _\Omega }{2}})$.
Therefore
\begin{align*}
\str _\Psi \big(\frac {N }{2} ( D (t) ^2 e ^{- \eth (t) ^2} ))
=& 2 ^{- \frac{ N _\Omega }{2}}
\str _\Psi \big( N( D (\frac{t}{2}) ^2 e ^{- \eth (\frac{t}{2}) ^2} e ^{- \eth (\frac{t}{2}) ^2}
- \varPi _0 (\Omega \varPi _0 ) ^2 e ^{(\Omega \varPi _0 )^2} e ^{(\Omega \varPi _0 )^2} ) \big) \\
=& 2 ^{- \frac{ N_\Omega }{2}} \str _\Psi \big( N( D (\frac{t}{2}) ^2 e ^{- \eth (\frac{t}{2}) ^2}
- \varPi _0 (\Omega \varPi _0 ) ^2 e ^{(\Omega \varPi _0 )^2} ) e ^{- \eth (\frac{t}{2}) ^2} \big) \\
&- 2 ^{- \frac{ N_\Omega }{2}} \str _\Psi \big( N\varPi _0 (\Omega \varPi _0 ) ^2 e ^{(\Omega \varPi _0 )^2}
(e ^{- \eth (\frac{t}{2}) ^2} - e ^{(\Omega \varPi _0 )^2} ) \big),
\end{align*}
which is also $ O (t ^{- \gamma ' })$ as $t \to \infty $ by similar arguments.

By the Sobolev embedding theorem (for the compact manifold $\rB$),
it follows that
$$ - F ^\wedge (t)
+ \frac{\str _\Psi (N \varPi _0 )}{2}
+ \big(\frac{\dim (\rZ) \rk (\rE) \str _\Psi (\varPi _0 )}{4} - \frac{\str _\Psi (N\varPi _0 )}{2} \big)
(1 - 2 t) e ^{-t} $$
and its all derivatives are $ O (t ^{- \gamma ' })$ uniformly.

Finally, since all derivatives of the $t$-integrand in Definition \ref{TorsionDfn} are $L ^1$,
derivatives of $\tau $ exist and equal differentiations under the $t$-integration sign.
Hence we conclude that the torsion $\tau $ is smooth.
\end{proof}

\begin{rem}
If $Z$ is $L^2$-acyclic and of determinant class (cf. \cite[Def. 6.3]{Schick;NonCptTorsionEst}), the analogue of Remark \ref{SquareNS} reads
$$ \int _0 ^\infty \| e ^{- t \varDelta } \| ^2 _{\HS 0} \frac{d t}{t}
= \int _0 ^\infty \| e ^{- t \varDelta } \| _\tau \frac{d t}{t} < \infty $$
(note that $\varPi _0 = 0 $ by hypothesis).
Unlike having positive Novikov-Shubin invariant,
the heat operator is not of determinant class in $\| \cdot \| _{\HS 0 }$.
\end{rem}

Given a power series $f (x) = \sum a _j x ^ j$. For clarity, let $h $ be the metric on
$\wedge ^ \bullet \rV \otimes \rE $  and we denote
\begin{align*}
f ( \nabla ^{\wedge ^\bullet \rV' \otimes \rE } , h)
&:= \str \Big( \sum _j a _j \big( \frac{1}{2}
(\nabla ^{\wedge ^\bullet \rV' \otimes \rE } - (\nabla ^{\wedge ^\bullet \rV' \otimes \rE })^*) \big)^{j} \Big)
\in \Gamma ^\infty (\wedge ^\bullet T ^* \rM ), \\
f ( \nabla ^{\wedge ^\bullet \rV' \otimes \rE } , h) _{H ^\bullet (\rZ , \rE)}
&:= \str _\Psi \Big( \sum _j a _j \big( \frac{1}{2} \varPi _0
(\nabla ^{\wedge ^\bullet \rV' _\flat \otimes \rE _\flat  }
- (\nabla ^{\wedge ^\bullet \rV _\flat' \otimes \rE _\flat })^*) \varPi _0 \big)^{j} \Big)
\in \Gamma ^\infty (\wedge ^\bullet T ^* \rB ).
\end{align*}
Note that the summations are only up to $\dim \rM$.

Let $TZ$  be the vertical tangent bundle of the fiber bundle $M\to B$ and recall that we have chosen a splitting of $TM$ and defined a Riemannian metric on $TM$. Let $P^{TZ}$ denote the projection from $TM$ to $TZ$. Let $\nabla^{TM}$ be the corresponding Levi-Civita connection on $TM$ and define $\nabla^{TZ}=P^{TZ}\nabla^{TM}P^{TZ}$, a connection on $TZ$. The restriction of $\nabla^{TZ}$ to a fiber coincides with the Levi-Civita connection of the fiber. Let $R^{TZ}$ be the curvature of $\nabla^{TZ}$.

For $N$ even, let ${\rm Pf}: \mathfrak{so}(N)\to \mathbb{R}$ denote the Pfaffian and put
\begin{align}
e\left(TZ,\nabla^{TZ}\right):=\left\{\begin{matrix}{\rm Pf} \left[{{R^{TZ}}\over{2\pi}}\right]& {\rm if}\ {\rm dim}(Z)\ {\rm is}\ {\rm even},\\0& {\rm if}\ {\rm dim}(Z)\ {\rm is}\ {\rm odd}.\end{matrix}\right.
\end{align}

A classical argument \cite{Bismut;AnaTorsion,Zhang;EtaTorsion,Schick;NonCptTorsionEst} then gives:
\begin{cor}
If $\dim \rZ = 2 n$ is even
one has the transgression formula
$$ d \tau (x)
= \int _{\rZ _x} \chi (x, z)
e (T \rZ,\nabla^{TZ} ) f ( \nabla ^{\wedge ^\bullet \rV' \otimes \rE } )
- f ( \nabla ^{\wedge ^\bullet \rV' \otimes \rE } ) _{H ^\bullet (\rZ , \rE)},$$
with $f (x) = x e ^{ x ^2 } $.
\end{cor}

Now let $ h _l $ be a family of $\rG$-invariant metrics on $\wedge ^\bullet \rV \otimes \rE $, $l \in [0, 1]$.
Define
$$ \tilde f (\nabla ^{\wedge ^\bullet \rV' \otimes \rE } , h _l)
:= \int _0 ^1 (2 \pi \sqrt{-1} ) ^{\frac{N_\Omega }{2}}
\str \Big( (h _l) ^{-1} \frac{d h _l}{d l} f ' ( \nabla ^{\wedge ^\bullet \rV' \otimes \rE } , h_l ) \Big) d l ,$$
and similarly for $\tilde f (\nabla ^{\wedge ^\bullet \rV' \otimes \rE } , h _l) _{H ^\bullet (\rZ , \rE )} $.
Note that $f ' ( \nabla ^{\wedge ^\bullet \rV' \otimes \rE } , h _l)$ uses the adjoint connection with respect to $h _l$.

Let $\widehat{e}\left(TZ,\nabla^{TZ,0},\nabla^{TZ,1}\right)\in Q^{M}/Q^{M,0}$ (cf. \cite{Bismut;AnaTorsion}) be the secondary class associated to the Euler class. Its representatives are forms of degree ${\rm dim}(Z)-1$ such that
\begin{align}
d\widehat{e}\left(TZ,\nabla^{TZ,0},\nabla^{TZ,1}\right)=e\left(TZ,\nabla^{TZ,1}\right)-e\left(TZ,\nabla^{TZ,0}\right).
\end{align}
If ${\rm dim}(Z)$ is odd, we take $\widehat{e}\left(TZ,\nabla^{TZ,0},\nabla^{TZ,1}\right)$ to be zero.

One has an anomaly formula \cite[Theorem 3.24]{Bismut;AnaTorsion}.

\begin{lem}
Modulo exact forms
\begin{align}
\label{Anomaly}
\tau _1 - \tau _0
=& \int_{ \rZ _x } \chi(x, z) \widehat{e} (TZ,\nabla^{TZ,0},\nabla^{TZ,1})
f (\nabla ^{\wedge ^\bullet \rV' \otimes \rE } , h _0) \\ \nonumber
&+ \int_{\rZ _x} \chi(x, z) e (T \rZ , \nabla^{TZ,1}) \tilde f (\nabla ^{\wedge ^\bullet \rV' \otimes \rE } , h _l)
- \tilde f (\nabla ^{\wedge ^\bullet \rV' _\flat \otimes \rE _\flat } , h _l) _{H ^\bullet (\rZ , \rE )}.
\end{align}
\end{lem}

In particular, the degree-0 part of Equation \eqref{Anomaly} is the anomaly formula for the $L ^2$-Ray-Singer
analytic torsion, which is a special case of \cite[Theorem 3.4]{Zhang;CoveringCheegerMuller}.

\begin{rem}
Let $\rZ _0 \to M _0 \to B$ be a fiber bundle with compact fiber $Z _0 $,
$\rZ \to \rM \to \rB$ be the normal covering of the fiber bundle $\rZ _0 \to \rM _0 \to \rB$.
Then one can define the Bismut-Lott and $L ^2$-analytic torsion form
$\tau _{\rM _0 \to \rB }, \tau _{\rM _0 \to \rB } \in \Gamma ^\infty (\wedge ^\bullet T^* \rB)$,
and one has the respective transgression formulas
\begin{align*}
d \tau _{\rM _0 \to \rB } &= \int _{\pi _0 ^{-1} (x)}
e (T \rZ _0 ,\nabla^{TZ_{0}}) f ( \nabla ^{\wedge ^\bullet \rV' _0 \otimes \rE _0 } )
- f ( \nabla ^{\wedge ^\bullet \rV' _0 \otimes \rE _0 } ) _{H ^\bullet (\rZ _0, \rE _0)} \\
d \tau _{\rM \to \rB } &= \int _{\rZ _x} \chi (x, z)
e (T \rZ ,\nabla^{TZ}) f ( \nabla ^{\wedge ^\bullet \rV' \otimes \rE } )
- f ( \nabla ^{\wedge ^\bullet \rV' \otimes \rE } ) _{H ^\bullet (\rZ , \rE)}.
\end{align*}
Suppose further that the DeRham cohomologies are trivial:
$$H ^\bullet (\rZ _0 , \rE|_{Z _0}) = H ^\bullet_{L ^ 2} (\rZ , \rE| _\rZ )= \{ 0 \}. $$
Then $d ( \tau _{\rM \to \rB } - \tau _{\rM _0 \to \rB } ) = 0 $.
Hence $\tau _{\rM  \to \rB } - \tau _{\rM _0 \to \rB }$ defines some class in the DeRham cohomology of $\rB$.
We also remark that this form was also mentioned in \cite[Remark 7.5]{Schick;NonCptTorsionEst},
as a weakly closed form.
\end{rem}

\begin{rem}
In our preprint \cite{SoSu}, we extended the methods in this paper to the case when $B/G$ is not a manifold. Regarding $B/G$ as a non-commutative space, in \cite{SoSu} we defined the non-commutative analytic torsion form and got a non-commutative Riemannian-Roch-Grothendieck theorem.
\end{rem}

\bibliographystyle{plain}

\end{document}